\documentclass[notitlepage,11pt,punct,oneside]{amsart}
\usepackage[a4paper,hmargin={2.54cm,2.54cm},vmargin={3.17cm,3.17cm}]{geometry}
\usepackage{amsmath,amssymb,amsthm}
\numberwithin{equation}{section}
\usepackage[runin]{abstract}
\usepackage{graphicx}
\usepackage{subfigure}
\usepackage{tikz}
\usepackage{float}

\usepackage[pdfborder={0 0 0},colorlinks=true,linkcolor=blue,CJKbookmarks=true]{hyperref}
\usepackage{mathrsfs}

\def\XXint#1#2#3{{\setbox0=\hbox{$#1{#2#3}{\int}$}
  \vcenter{\hbox{$#2#3$}}\kern-.5\wd0}}

\newtheorem{definition}{\hspace{2em}Definition}[section]
\newtheorem{theorem}[definition]{\hspace{2em}Theorem}

\newtheorem{lemma}[definition]{\hspace{2em}Lemma}
\newtheorem{proposition}[definition]{\hspace{2em}Proposition}

\newtheorem{remark}{\hspace{2em}Remark}%[section]
\newtheorem{example}[definition]{\hspace{2em}Example}

\begin{document}
\title{Asymptotic behavior of $\beta$-polygon flows}

\author{David Glickenstein}\thanks{DG was partly supported by NSF grant DMS 0748283.}
\address{Department of Mathematics\\
University of Arizona\\
Tucson, AZ 85721}
\email{glickenstein@math.arizona.edu}

\author{Jinjin Liang}
\address{Department of Mathematics\\
University of Arizona\\
Tucson, AZ 85721}
\email{jliang@math.arizona.edu}


\newdimen\R
\R=0.8cm
\newcommand{\real}{\mathbb{R}}
%\renewcommand{\thepage}{\roman{page}}
%\setcounter{page}{1}
%\tableofcontents\clearpage

\maketitle
\begin{abstract}
   In this article we investigate a family of nonlinear evolutions of polygons in the plane called the
   $\beta$-polygon flow and obtain some results analogous to results for the smooth curve shortening flow:
   (1) any planar polygon shrinks to a point and (2) a regular polygon with five or more vertices is
   asymptotically stable in the sense that nearby polygons shrink to points that rescale to a regular polygon.
   In dimension four we show that the shape of a square is locally stable under perturbations along a hypersurface
   of all possible perturbations.
   Furthermore, we are able to show that under a lower bound on angles there exists a rescaled
   sequence extracted from the evolution that converges to a limiting polygon that is a self-similar
   solution of the flow. The last result uses a monotonicity formula analogous to Huisken's for the curve shortening flow.
\end{abstract}

\section{Introduction}

The Gage-Grayson-Hamilton Theorem (\cite{gage,gray}) states that any embedded plane curve converges
to a round point in an asymptotically self-similar manner under the motion by its curvature.
One interesting and still open question is whether one can find a discrete version of the
curve shortening flow such that any embedded polygon contracts to a regular point in an asymptotically
self-similar manner. Several approaches to this have been suggested, such as flow by the generalized gradient
flow of the length functional \cite{Naka,dziuk} and flow by the Menger curvature \cite{jeck}.
 However, even locally, none of these flows gives an affirmative answer to the above question.
 It seems that these flows (\cite{Naka, dziuk, jeck}) may cease to be defined
when one of the edge lengths becomes zero, which can happen, for instance for a long, skinny rectangle.
In this paper, we consider a slightly different flow that does not have a problem when an edge length becomes zero.

We consider a family of nonlinear evolutions of polygons. 
\begin{definition}
A family of polygons $X(t)=\left(X_0,\ldots,X_{N-1}\right)$ (see Definition \ref{defpolygon}) evolves by the \emph{$\beta$-polgyon flow} if it satisfies
\begin{equation}\label{flowmatrix}
\frac{dX_j}{dt}=l_j^\beta (X_{j+1}-X_j)+l_{j-1}^\beta (X_{j-1}-X_j), %\quad j=0,\ldots,N-1,
\end{equation}
where $\beta \geq 0$ and $l_j = \left| X_{j+1}-X_j \right|$ for the parameters $j=0,\dots,N-1$ 
considered module $N$.
\end{definition}

% (\ref{flowmatrix}) with the parameter $\beta>0$ that
%we call the $\beta$-polygon flow that is the gradient flow of the $L^{2+\beta}$ norm of the edge lengths.  
In
\cite{dav}, Chow and Glickenstein consider the system (\ref{flowmatrix}) when $\beta=0$. In this case,
(\ref{flowmatrix}) turns out to be a linear system. The main results obtained in \cite{dav} are that
 the flow shrinks any polygon to a point and the asymptotic shape is affinely-regular if the initial polygon
  is not orthogonal to the regular polygon. The linear flow has the advantage that there is no singularity
  before the polygon extinguishes. A disadvantage is that the space of affinely-regular polygons
   is a big space;
  for example, all triangles and parallelograms are affinely-regular. Therefore, in the end of \cite{dav}, the
  authors ask whether the nonlinear system (\ref{flowmatrix}) flows a polygon asymptotically to
  a regular polygon. We are able to give a partial answer to this question. We prove that the $\beta$-polygon
  flow converges to a self-similar solution.

%\begin{theorem}\label{stableshapethm5}
%  Assume $N\geq5$. Under the $\beta$-polygon flow, any regular $N$-gon is
%  asymptotically stable under arbitrarily small perturbations;
%that is, for any regular polygon there is a neighborhood such that polygons in that neighborhood will converge to a regular polygon under the $\beta$-polygon flow if appropriately rescaled.
%\end{theorem}

\begin{theorem}\label{contractsimilar}
  Let $X(t)$ be the solution of the $\beta$-polygon flow for $t\in[0,\infty)$. Assume the angle bound (\ref{anglelowerbd}) is
  satisfied and suppose $X(t)\rightarrow x_0$ as $t\rightarrow\infty$. Then for any sequence $c_k\nearrow\infty$
  there exists a subsequence still denoted by $c_k$  such that the following rescaled polygons converge to a polygon
  that contracts self-similarly:
  \begin{displaymath}
    c_k\left[X(c_k^\beta\tau)-x_0\right]\rightarrow Y(\tau),
  \end{displaymath}
  where $Y(\tau)$ is a self-similar solution for $\tau>0$.
\end{theorem}

The convex regular polygons are self-similar solutions. We are furthermore able to prove these are stable in the
following theorems.

\begin{theorem}\label{stableshapethm5}
  Assume $N\geq5$. Under the $\beta$-polygon flow, any regular $N$-gon shrinks to a point
  and is asymptotically stable in the sense that there is a neighborhood such that polygons in that
  neighborhood will converge to a regular polygon under the $\beta$-polygon flow if appropriately rescaled.
\end{theorem}

\begin{theorem}\label{stableshapethm4}
  When $N=4,$ the shape of square is locally stable on a 7-dimensional hypersurface $\mathcal{W}'$ under the
  $\beta$-polygon flow.
\end{theorem}

%The outline of this article is as follows. In section \ref{secexistunique}, we establish the long time existence and uniqueness of the initial value problem of (\ref{flowmatrix}). In section \ref{sectria}, we construct a Lyapunov function to show that any triangle would converge to a regular triangle. In section \ref{seclocal}, first we turn the self-similar solution (regular $N$-gon) into an equilibrium point of a scaling system, by linearizing the scaling system and using the center manifold theorem, we prove that the regular shape is asymptotic stable to arbitrary small perturbations under the flow (\ref{flowmatrix}) when $N\geq 5$. Furthermore, when $N=4$ we have prove that the shape of square is locally stable at most at a hypersurface under the evolution. In section \ref{secglobal}, inspired by Huisken's monotonicity formula \cite{hui}, we construct an entropy functional and apply the standard blow up argument in the geometric flows (see \cite{schn},\cite{hui}, for instance), we obtain that for any rescaled sequence extracted from the evolution, there exists a subsequence converge to a polygon that contracts self-similarly.

The outline of this article is as follows. In Section \ref{secexistunique}, we establish the long time existence and
uniqueness of the initial value problem for the $\beta$-polygon flow. In Section \ref{sectria}, we construct a Lyapunov
function to show that any triangle would converge to a regular triangle. In Section \ref{secselfsim}, inspired by
Huisken's monotonicity formula \cite{hui}, we have the global stability result, Theorem \ref{contractsimilar}. In
Section \ref{seclocal}, we obtain the local stability of the $\beta$-polygon flow (\ref{flowmatrix}) in Theorems
\ref{stableshapethm5} and \ref{stableshapethm4}.

\section{Existence and basic properties}\label{secexistunique}

In this section we will describe the $\beta$-polygon flow and give basic
properties of it. First we give a definition of a polygon.

\begin{definition}\label{defpolygon}
  An $N$-gon, or polygon, $X$ in the Euclidean plane is an ordered $N$-tuple of points in the plane, $X=(X_0,\cdots,X_{N-1})$. Note that the index of the points will always be considered modulo $N$.
\end{definition}
  The points of the polygon are called \emph{vertices} and the line segments joining consecutive vertices are called \emph{edges}. The geometry of the polygon is determined by the following quantities.

\begin{definition}  \label{defangle}
  The \emph{length} of an edge, denoted $l_j$ is defined to be the distance between the adjacent vertices $X_j$ and $X_{j+1}$. The angle $\theta_j$ at vertex $X_j$ is defined to be the angle such that rotating the unit vector $\overrightarrow{X_jX_{j+1}}/|\overrightarrow{X_jX_{j+1}}|$ an angle of $\theta_j$ in the counterclockwise direction gives the vector $\overrightarrow{X_jX_{j-1}}/|\overrightarrow{X_jX_{j-1}}|$.
\end{definition}
%\begin{remark}
%We say rotate some vector by some angle $\theta$ (counter)-clockwisely, it means, first, we translate the vector so that it starts at the origin, then rotate it around the origin by $\theta$ (counter)-clockwisely.
%\end{remark}
%\begin{remark}
%$\theta'$, $\sin\theta$ and $\cos\theta$ are uniquely determined according to this definition.
%%Therefore, through out this paper, our calculation only involves $d\theta/dt,\sin\theta$ and $\cos\theta$.
%\end{remark}
Note that we are using $\theta_j$ to denote the interior angle of a polygon. In
some related work, the angle is defined to be the exterior angle, and would
have the value of $\pi - \theta_j$. We are also assuming that consecutive
vertices are not equal, in which case we would not be able to define angle.

We have a natural identification between an $N$-gon in $\mathbb{R}^2$ and an $N$-vector in $\mathbb{C}^n.$ In particular, we can write the vertex $X_j=(x_j,y_j)$ as $X_j=x_j+iy_j$ where $i=\sqrt{-1}).$
Sometimes it is convenient to write the polygon as a $N \times 2$ matrix:
\[
X=\left(
	\begin{array} [c]{cc}
	x_{0} & y_{0}\\
	x_{1} & y_{1}\\
	\vdots & \vdots\\
	x_{N-1} & y_{N-1}
\end{array}
\right).
\]
In this case, any two-by-two matrix $M$ can act on $X$ on the right by matrix multiplication, $XM$.
We will consider actions by a Euclidean isometry $E$ from the right as well; the
rotational part acts by matrix multiplication on the right and the translational part acts by adding a matrix with all rows equal.

Let $X=(X_0,\cdots,X_{N-1})$ be a planar $N$-gon. Consider the following energy functional on $X$:
\begin{equation}\label{energey}
F_\alpha(X)=\frac{1}{\alpha}\sum_{j=0}^{N-1}|X_{j+1}-X_j|^\alpha ,
\end{equation} where the indices, as usual, are taken modulo $N$. We can then compute the variation of this functional.
If $dX_j/dt=Y_j,$ then
\begin{equation}
\frac{d}{dt}F_\alpha(X)=-\sum_{j=0}^{n-1}\left(\frac{X_{j+1}-X_j}{|X_{j+1}-X_j|^{2-\alpha}}+\frac{X_{j-1}-X_j}{|X_{j-1}-X_j|^{2-\alpha}}\right)\cdot Y_j.
\end{equation}
The negative gradient flow of $F_\alpha$ is therefore
\begin{equation}\label{flow}
\frac{dX_j}{dt}=\frac{X_{j+1}-X_j}{|X_{j+1}-X_j|^{2-\alpha}}+\frac{X_{j-1}-X_j}{|X_{j-1}-X_j|^{2-\alpha}}.
\end{equation}

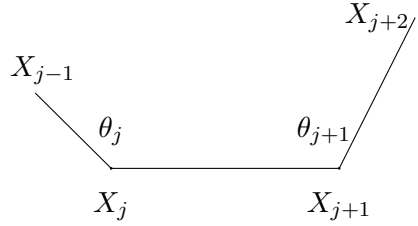
\begin{figure}[htbp]
\centering
\begin{tikzpicture}
  \draw  (-1,1)--(0,0);\draw (0,0)--(3,0);\draw (3,0)-- (4,2);
  \fill (3,0) circle[radius=0.5pt];
\fill (0,0) circle[radius=0.5pt];

\node (xj) at (0,-0.5){$X_j$};
\node (xj1) at (3,-0.5){$X_{j+1}$};
\node (xj-1) at (-0.9,1.3){$X_{j-1}$};
\node (xj+2) at (3.5,2){$X_{j+2}$};

\node (thej) at (0,0.5){$\theta_j$};
\node (thej1) at (2.8,0.5){$\theta_{j+1}$};

\end{tikzpicture}
\caption{Part of a polygon}\label{lenang}
\end{figure}

%\begin{definition}\label{defangle}
%  We say $\theta_j$ is the angle at the vertex $X_j$ if and only if by rotating counterclockwise $\overrightarrow{X_jX_{j+1}}/|\overrightarrow{X_jX_{j+1}}|$ a angle of $\theta_j$ we get $\overrightarrow{X_jX_{j-1}}/|\overrightarrow{X_jX_{j-1}}|$.
%\end{definition}
%
%\begin{remark}
%We say rotate some vector $V$ by some angle $\theta$ (counter)-clockwise, it means, first, we translate the vector $V$ so that it starts at the origin, then rotate it around the origin by $\theta$ (counter)-clockwise.
%\end{remark}
%
%\begin{remark}
%$\theta'$, $\sin\theta$ and $\cos\theta$ are unique determined according to this definition.
%%Therefore, through out this paper, our calculation only involves $d\theta/dt,\sin\theta$ and $\cos\theta$.
%\end{remark}

In this paper, we consider the evolution of the planar polygons under the flow (\ref{flow}) in the case $\alpha=\beta+2$ where $\beta>0.$ Recalling that $l_j$ denotes the edge length between $X_j$ and $X_{j+1}$, we can rewrite (\ref{flow}) as (\ref{flowmatrix}). The matrix form of (\ref{flowmatrix}) is
\begin{equation}\label{flowmatrix2}
\frac{d}{dt}X=M_{X}X,
\end{equation} in which $$M_{X}=\left(
                       \begin{array}{cccccc}
                         -(l_0^\beta+l_{n-1}^\beta) & l_0^\beta & 0 & \cdots & 0 & l_{N-1}^\beta \\
                         l_0^\beta & -(l_0^\beta+l_1^\beta) & l_1^\beta & 0 & \ddots & 0 \\
                         0 & l_1^\beta & -(l_1^\beta+l_2^\beta) & l_2^\beta & 0 & \vdots \\
                         \vdots & 0 & \ddots & \ddots & \ddots & 0 \\
                         0 & \ddots & 0 & l_{N-3}^\beta & -(l_{N-3}^\beta+l_{N-2}^\beta) & l_{N-2}^\beta \\
                         l_{N-1}^\beta & 0 & \cdots & 0 & l_{N-2}^\beta & -(l_{N-2}^\beta+l_{N-1}^\beta) \\
                       \end{array}
                     \right).$$
We may refer to either of the equivalent systems (\ref{flowmatrix}) or (\ref{flowmatrix2}) as the $\beta$-polygon flow.
\begin{remark}
 The matrix $M_X$ has the form of a weighted graph Laplacian on the $N$-cycle $X$, where each edge $\overrightarrow{X_kX_{k+1}}$ has weight $l_k^\beta$. Therefore, (\ref{flowmatrix}) can be considered a type of heat equation.
\end{remark}

We describe some basic properties of $M_X$.
\begin{proposition}\label{prop:M properties}
Let $X$ be a polygon, $c >0$, and $E$ be a Euclidean isometry of the plane. We denote the action of $E$ on $X$ by $XE$ and use $\vec{1}$ to denote the vector of all ones. Then the following are true:
\begin{enumerate}
\item $M_{XE}=M_X$.
\item $M_{cX} = c^\beta M_X$.
\item $M_X \vec{1}=0$.
\end{enumerate}
\end{proposition}

\begin{proof}
The first follows from the fact that $M_X$ uses only the edge lengths and not the points themselves,
so the matrix is unchanged by Euclidean transformations. The second is a scaling property that is
easily checked and the third comes from the form of the matrix $M_X$.
\end{proof}

%\begin{definition}\label{defdilation}
%  For any $N$-gon $X=(X_0,\cdots, X_{N-1})\in\mathbb{C}^N$, let $\mathcal{B}$  be the set of Euclidean plane isometries, for any positive number $c$, define $(c\mathcal{B},1/c^\beta)$ to be the following affine transformations in space and time
%  \begin{equation}\label{dilationsim}
%    Y=cXB,\tau=\frac{1}{c^\beta} t,\quad B\in\mathcal{B}.
%  \end{equation}
%\end{definition}
%\begin{remark}\label{invariantrigid}
%  The equation (\ref{flowmatrix}) is invariant under any rigid motions: rotations, translations, conjugations, etc.
%\end{remark}
This leads to the following important invariant property of
%(\ref{flowmatrix})
the $\beta$-polygon flow, which follows from the previous proposition.

\begin{lemma}\label{leminvariant}
  The $\beta$-polygon flow is invariant in the following way: if $c>0$, $E$ is a Euclidean
  transformation of the plane, and $\tau=\frac{1}{c^\beta}t$ then
  \[
  \frac{d}{d\tau}\left( cXE \right) = M_{cXE}(cXE)
  \]
\end{lemma}
%\begin{proof}
%  We identify $\mathbb{C}$ with $\mathbb{R}^2$. Then the $N$-gon $X\in\mathbb{R}^{2N}$ is a vector of points $X_j=(x_j,y_j)$:
%  \begin{displaymath}
%    X=\begin{pmatrix}
%        X_0 \\
%        X_1 \\
%        \vdots \\
%        X_{N-1} \\
%      \end{pmatrix}=\begin{pmatrix}
%                      x_0 & y_0 \\
%                     x_1 & y_1 \\
%                      \vdots & \vdots \\
%                      x_{N-1} & y_{N-1} \\
%                    \end{pmatrix}
%  \end{displaymath}
%  and let $\mathcal{B}$ to be the set of isometries of $\mathbb{R}^2$. For any number $k$ we let $\vec{k}$ denote the $N \times 1$ column vector each of whose entries is $k$. Note that $M_X\vec{k}=0.$
%
%  First let us look at the case $c=1$. Given $B\in\mathcal{B}$, then $B=L+(a,b)$, where $L$ is a $2 \times 2$ matrix that preserve the edge-length of $X$ and $(a,b)$ is a point in $\mathbb{R}^2$, and we have
%\begin{equation}\label{affinepreedge2}
%\frac{d}{dt}(XL+(\vec{a},\vec{b}))=\frac{dX}{dt}L=M_X(X)L=M_{X}(XL+(\vec{a},\vec{b}))=M_{XL+(\vec{a},\vec{b})}(XL+(\vec{a},\vec{b})).
%\end{equation}
%
%Now for arbitrary $c>0$, we have
%\begin{displaymath}
%  \frac{d }{d\tau}(cXB)=c^{1+\beta}\frac{d}{dt}(XB)=c^{1+\beta}M_{XB}(XB)=M_{cXB}(cXB).
%\end{displaymath}
%
%\end{proof}

\begin{remark}
  This invariant property together with the uniqueness Theorem \ref{exist} below  implies that similar polygons evolve in a similar manner under %(\ref{flowmatrix})
  the $\beta$-polygon flow. Let $X$ and $Y$ be solutions of (\ref{flowmatrix}) such that $Y(0)=cX(0)E$ for some number $c$ and some Euclidean isometry $E$. Then $Y(\tau(t))=cX(t)E$.
\end{remark}
%\begin{remark}
%  We shall see in Section \ref{evoquadrilateral} that the rigid motion is a proper subset of the length-preserving affine transforms associate to a square.
%\end{remark}
We can describe the fixed points explicitly.
\begin{proposition}\label{prop:fixed points}
The fixed points of the flow are precisely polygons of the form $X=(X_0,\ldots, X_{N-1})$ such that $X_0=X_1=\cdots = X_{N-1}$. We call such polygons \emph{points}.
\end{proposition}

It will be important to control the $(2+\beta)$-norm of the polygon under the flow
in order to use appropriate compactness theorems.
\begin{definition}
  If $X=(X_0,\cdots,X_{N-1})\in\mathbb{C}^N$, for any $p \geq 1$, we define the $p$-norm of $X$ by
  \begin{equation}\label{lpnrom}
    \Vert X\Vert_p=\left(\sum_{k=0}^{N-1}|X_k|^p\right)^{1/p}.
  \end{equation}
\end{definition}
We now have the following a priori bound on the $2+\beta$-norm.
\begin{lemma}[A priori bound]\label{bdlem}
  Let $\alpha=2+\beta$ with $\beta>0$ and $\mathbf{Q}=(Q,\cdots,Q)\in \mathbb{C}^N$. If $X(t)$ is the solution of the initial value problem for (\ref{flowmatrix}), then the $\alpha$-norm of $X-\mathbf{Q}$ is monotonicity decreasing, i.e., $\Vert X(t)-\mathbf{Q}\Vert_\alpha\leq \Vert X(\tau)-\mathbf{Q}\Vert_\alpha$ for $t>\tau.$
\end{lemma}
\begin{proof}
  A direct calculation gives:
  \begin{displaymath}
  \begin{split}
  &  \frac{d}{dt}\frac{1}{\alpha}\Vert X-\mathbf{Q}\Vert_\alpha^\alpha
%  = \sum_{j=0}^{N-1}|X_j-Q|^\beta (X_j-Q)\cdot\frac{d X_j}{dt}
  \\
%      & =\sum_{j=0}^{N-1}|X_j-Q|^\beta (X_j-Q)\cdot\left[l_j^\beta(X_{j+1}-X_j)+l_{j-1}^\beta(X_{j-1}-X_j)\right]\\
%      &=\sum_{j=0}^{N-1}\left[|X_j-Q|^\beta l_j^\beta (X_j-Q)\cdot(X_{j+1}-X_j)+|X_{j+1}-Q|^\beta l_j^\beta (X_{j+1}-Q)\cdot(X_j-X_{j+1})\right]\\
%      &=-\sum_{j=0}^{N-1}l_j^\beta\left[|X_{j+1}-Q|^\beta (X_{j+1}-Q)-|X_j-Q|^\beta (X_j-Q)\right]\cdot[X_{j+1}-Q-(X_j-Q)]\\
      &=-\sum_{j=0}^{N-1}l_j^\beta\left[|X_{j+1}-Q|^{\beta+2}+|X_j-Q|^{\beta+2}-(|X_{j+1}-Q|^\beta+|X_j-Q|^\beta)(X_j-Q)\cdot( X_{j+1}-Q)\right]\\
	 &\leq -\sum_{j=0}^{N-1}l_j^\beta\left[(|X_{j+1}-Q|^{\beta+1}-|X_j-Q|^{\beta+1})(|X_{j+1}-Q|-|X_j-Q|)\right] \\
	 &\leq 0
  \end{split}
\end{displaymath}
where the first inequality comes from Cauchy-Schwarz.
%Since \begin{displaymath}
%        \begin{split}
%        & |X_{j+1}-Q|^{\beta+2}+|X_j-Q|^{\beta+2}-(|X_{j+1}-Q|^\beta+|X_j-Q|^\beta)(X_j-Q)\cdot (X_{j+1}-Q)\\
%          &\geq|X_{j+1}-Q|^{\beta+2}+|X_j-Q|^{\beta+2}-(|X_{j+1}-Q|^\beta+|X_j-Q|^\beta)|X_j-Q| |(X_{j+1}-Q)|\\
%      &=(|X_{j+1}-Q|^{\beta+1}-|X_j-Q|^{\beta+1})(|X_{j+1}-Q|-|X_j-Q|)\geq 0.
%        \end{split}
%      \end{displaymath}
%       Therefore, we have $$\frac{d}{dt}\Vert X-\mathbf{Q}\Vert_\alpha^\alpha\leq 0,\qquad \forall t<T,$$ which implies that $$\frac{d}{dt}\Vert X-\mathbf{Q}\Vert_\alpha\leq 0,\qquad \forall t<T,$$ where $T$ is the maximal existence time of (\ref{flowmatrix}).
The result follows.
\end{proof}

We are now able to prove that the flow exists for all time and shrinks to its
center of mass.

\begin{theorem}[Long time existence/uniqueness]\label{exist} For any initial polygon $X=(X_0,\ldots, X_{N-1})$, there is a unique solution to (\ref{flowmatrix}) for all $t>0$ and the solution converges to the center of mass for $X$, $\frac1N \sum_{j=0}^{N-1} X_j$, as $t \to \infty$. 　
\end{theorem}
\begin{proof}
  The standard theory in the ordinary differential equations says the solution $X(t)$ of (\ref{flowmatrix}) exists at $[0,T)$ for some $T>0.$ Since the system has the
  form $\frac{d}{dt}X =F(X)$ where $dF$ is Lipschitz,
  the solution is also unique by the standard theory.

Applying Lemma \ref{bdlem} with $Q$ chosen to be the origin, we have $X(t)\subseteq B(R)$ for some closed Euclidean ball $B(R)$ for all $t>0$.  However, since $M_{X}:\mathbb{C}^{N}\rightarrow\mathbb{C}^{N}$ is a continuous function on $\mathbb{C}^{N}$, the extension theorem (p 12 of \cite{Hartman}) says that $X(t)$ would become unbounded as $t\rightarrow T$ if $T<\infty.$ Hence $T=\infty.$

%Next, it's clear that the set $\{\overrightarrow{X}\in\mathbb{C}^{n}|\frac{d}{dt}F_\alpha(\overrightarrow{X})=0\}=\{X_0=\cdots=X_{n-1}|\overrightarrow{X}\in\mathbb{C}^{n}\}:=U.$

In order to show that the flow converges to the center of mass, we first show that
 a subsequence converges to a point. Since $F_\alpha(X(t))$ is decreasing and
  bounded from below, as $t\rightarrow \infty$ we have that
  $\frac{d}{dt}F_\alpha (X(t)) \rightarrow 0$. Hence, by
  Proposition \ref{prop:fixed points}, any subsequence that converges
  to a polygon, converges to a point.

  Since $X(t)$ is in a bounded set, there is a subsequence $t_k$ such that
  $X(t_k)$ converges to a point, $Z$. By Lemma \ref{bdlem}, we have that
  $\Vert X(t)-Z\Vert_\alpha$ is decreasing in $t$ and since a subsequence converges
  to zero, we must have that $\Vert X(t)-Z\Vert_\alpha$ converges to zero.

Finally, the flow (\ref{flowmatrix}) preserves the center of mass of $X$, i.e., $$\frac{d}{dt}\frac{1}{N}\sum_{i=0}^{N-1}X_i=0.$$ Therefore, $Z$ must be the center of mass of $X$.

\end{proof}

A direct calculation gives the evolution of the edge lengths and angles.
\begin{lemma}
  Under the flow (\ref{flowmatrix}), we have
  \begin{equation}\label{length}
    \frac{dl_j}{dt}=-2l_j^{\beta+1}-l_{j+1}^{\beta+1}\cos\theta_{j+1}-l_{j-1}^{\beta+1}\cos\theta_j,
  \end{equation}
  \begin{equation}\label{angle}
    \frac{d{\theta}_j}{dt}=\frac{1}{l_jl_{j-1}}\left[(l_{j-1}^{\beta+2}+l_j^{\beta+2})\sin\theta_j-l_{j+1}^{\beta+1}l_{j-1}\sin\theta_{j+1}-l_{j-2}^{\beta+1}l_j\sin\theta_{j-1}\right],
  \end{equation} for $j=0,...,N-1.$
\end{lemma}
\begin{proof}
  %In most of the calculation in this paper, we identified $\mathbb{C}$ with $\mathbb{R}^2$.
  We have
  \begin{equation}\label{verdiff}
     \frac{d{X}_{j+1}}{dt}-\frac{dX_j}{dt}=l_{j+1}^\beta(X_{j+2}-X_{j+1})-2l_j^\beta(X_{j+1}-X_j)+l_{j-1}^\beta(X_j-X_{j-1}).
  \end{equation}
Hence,
  \begin{displaymath}
    \begin{split}
       \frac{d}{dt}l_j& =\frac{1}{l_j}(X_{j+1}-X_j)\cdot\left(\frac{d{X}_{j+1}}{dt}-\frac{dX_j}{dt}\right) \\
 %       & =\frac{1}{l_j}(X_{j+1}-X_j)\cdot\left[l_{j+1}^\beta(X_{j+2}-X_{j+1})-2l_j^\beta(X_{j+1}-X_j)+l_{j-1}^\beta(X_j-X_{j-1})\right]\\
        &=-2l_j^{\beta+1}-l_{j+1}^{\beta+1}\cos\theta_{j+1}-l_{j-1}^{\beta+1}\cos\theta_j.
    \end{split}
  \end{displaymath}
  Since $$\cos\theta_j=\frac{(X_{j+1}-X_j)\cdot(X_{j-1}-X_j)}{l_jl_{j-1}},$$ differentiating, we get
  \begin{displaymath}
%    \begin{split}
      -\sin\theta_j\frac{d{\theta}_j}{dt}  =
%      [(l_jl_{j-1})^{-1}]'l_jl_{j-1}\cos\theta_j \\
%        & +(l_jl_{j-1})^{-1}[(\dot{X}_{j+1}-\dot{X_j})\cdot(X_{j-1}-X_j)+(X_{j+1}-X_j)\cdot(\dot{X}_{j-1}-\dot{X}_j)]\\
%        & =I+II.
%    \end{split}
%  \end{displaymath} Substitute (\ref{length}) into $I$, we have
%  \begin{displaymath}
%    \begin{split}
%      I & =\frac{1}{l_jl_{j-1}}(l_{j+1}^{\beta+1}l_{j-1}\cos\theta_{j+1}+2l_j^{\beta+1}l_{j-1}+l_{j-1}^{\beta+2}\cos\theta_j)\cos\theta_j \\
%        & +\frac{1}{l_jl_{j-1}}(l_j^{\beta+2}\cos\theta_j+2l_jl_{j-1}^{\beta+1}+l_{j-2}^{\beta+1}l_j\cos\theta_{j-1})\cos\theta_j,
%    \end{split}
%  \end{displaymath}
%  Substitute (\ref{verdiff}) into $II$, we obtain
%  \begin{displaymath}
%    \begin{split}
%      II & =\frac{1}{l_jl_{j-1}}[-l_{j+1}^\beta(X_{j+2}-X_{j+1})\cdot(X_j-X_{j-1})-2l_j^{\beta+1}l_{j-1}\cos\theta_j-l_{j-1}^{\beta+2} \\
%        & -l_j^{\beta+2}-2l_{j-1}^{\beta+1}l_j\cos\theta_j-l_{j-2}^\beta(X_{j+1}-X_j)\cdot(X_{j-1}-X_{j-2})]
%    \end{split}
%  \end{displaymath}
%  Now, we claim that $$(X_{j+1}-X_j)\cdot(X_{j-1}-X_{j-2})=l_jl_{j-2}\cos(\theta_j+\theta_{j-1}).$$
%  To see this, definition \ref{defangle} says by rotating $X_{j+1}-X_j$ counterclockwise an angle of $\pi+\theta_j$, we get a vector towards $X_j-X_{j-1}$; rotate this new vector counterclockwise by another $\pi+\theta_{j-1}$ we get a vector towards $X_{j-1}-X_{j-2}$, completing the proof of this claim.
%
%  Plug this into $I+II$, we have
%  \begin{displaymath}
%    I+II=
    \frac{1}{l_jl_{j-1}}\left[-(l_{j-1}^{\beta+2}+l_j^{\beta+2})\sin^2\theta_j+l_{j+1}^{\beta+1}l_{j-1}\sin\theta_j\sin\theta_{j+1}+l_{j-2}^{\beta+1}l_j\sin\theta_j\sin\theta_{j-1}\right],
  \end{displaymath} which gives (\ref{angle}).
\end{proof}
\begin{remark}
One tricky part in the previous calculation is to show that $$(X_{j+1}-X_j)\cdot(X_{j-1}-X_{j-2})=l_jl_{j-2}\cos(\theta_j+\theta_{j-1}).$$
  To see this, recall that Definition \ref{defangle} says that by rotating $X_{j+1}-X_j$ counterclockwise an angle of $\pi+\theta_j$, we get a vector in the direction of $X_j-X_{j-1}$. Rotate this new vector counterclockwise by another $\pi+\theta_{j-1}$ and we get a vector in the direction of $X_{j-1}-X_{j-2}$.
\end{remark}

We close this section with numerical examples of the $\beta$-polygon flow on a heptagon and on a quadrilaterals. They indicate that the regular heptagon may be stable and that the square may be semistable, as described in Theorems \ref{stableshapethm5} and \ref{stableshapethm4}.

\begin{example}
 Figures \ref{hepta} and \ref{heptaseq} show, for the case $\beta=1, c_k=10^k, \tau=1$ and $N=7$, the evolution of a heptagon converging to a regular heptagon. Indeed, we start from some heptagon $X_0$ Figure \ref{hepta:a}, and evolve it under the flow (\ref{flowmatrix}) until time $\tau=1$ to obtain $X(1)$ in Figure \ref{heptaseq:a}. We use the rescaled heptagon $10X(1)$ as our new initial data and continue to evolve it under the flow (\ref{flowmatrix}) until the time $\tau=1$ to get $10X(10^1\cdot 1)$ in Figure \ref{heptaseq:b}. We rescale it by 10 again and repeat this process 6 times to obtain  Figures \ref{hepta} and \ref{heptaseq}. Comparing the polygon $10^5X(10^5)$ in Figure \ref{heptaseq:f} with the regular heptagon, we find
  \begin{displaymath}
      \sum_{i=0}^{6}\left(\theta_i-\frac{5\pi}{7}\right)^2=0.0252069, \quad  \sum_{i=0}^{6}\left(\frac{l_i}{l_{i+1}}-1\right)^2 =0.0107429,
  \end{displaymath}
  where $\theta_i$ and $l_i$ denote the angle and edge-length of $10^5X(10^5)$, respectively. It appears that the two errors become small and the polygons obtained in this process are converging to a regular heptagon.

  \begin{figure}[H]
  \centering
  \subfigure[]{
    \label{hepta:a} %% label for first subfigure
    \includegraphics[width=1.5in]{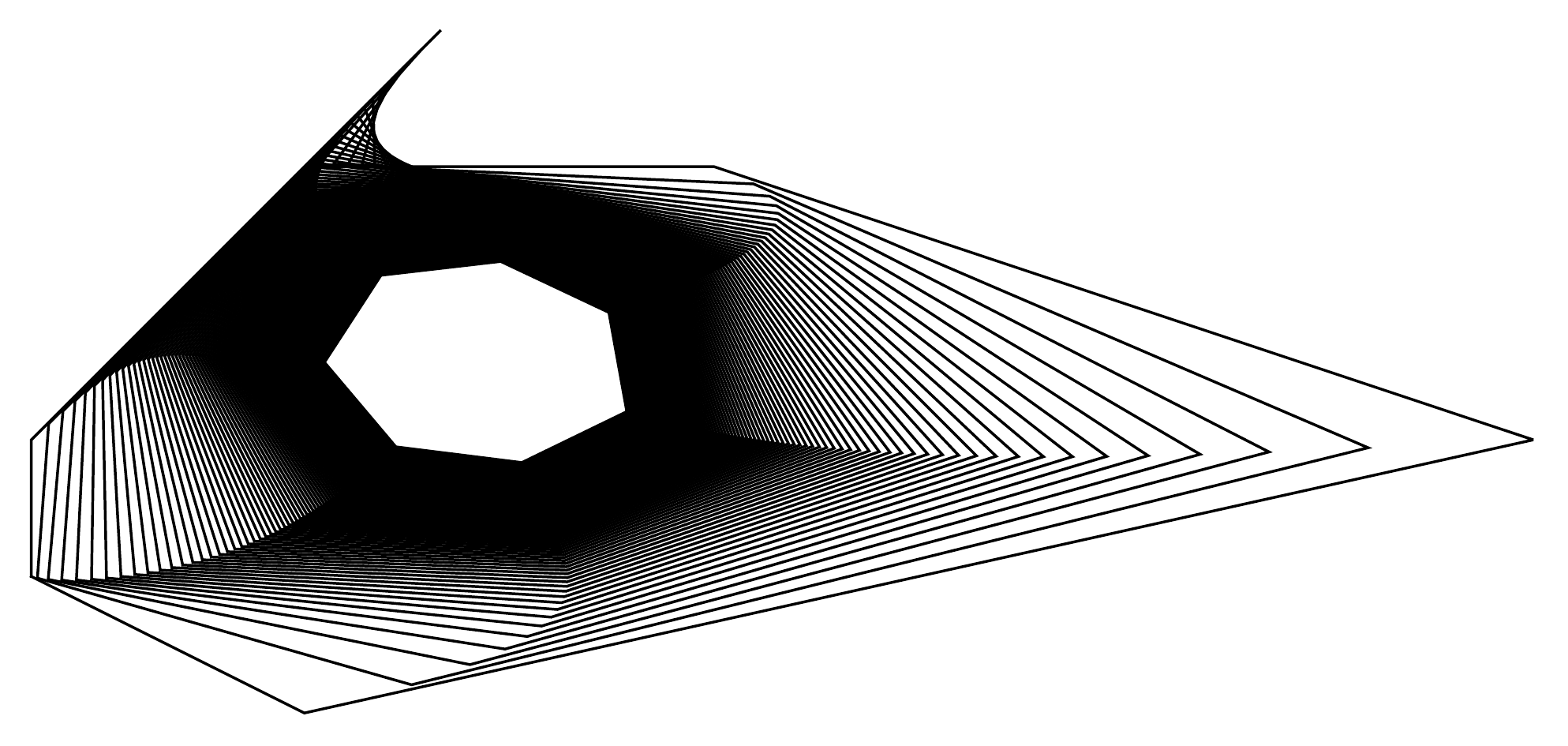}}
    \hspace{0.5in}
 \subfigure[]{
    \label{hepta:b} %% label for second subfigure
    \includegraphics[width=1.5in]{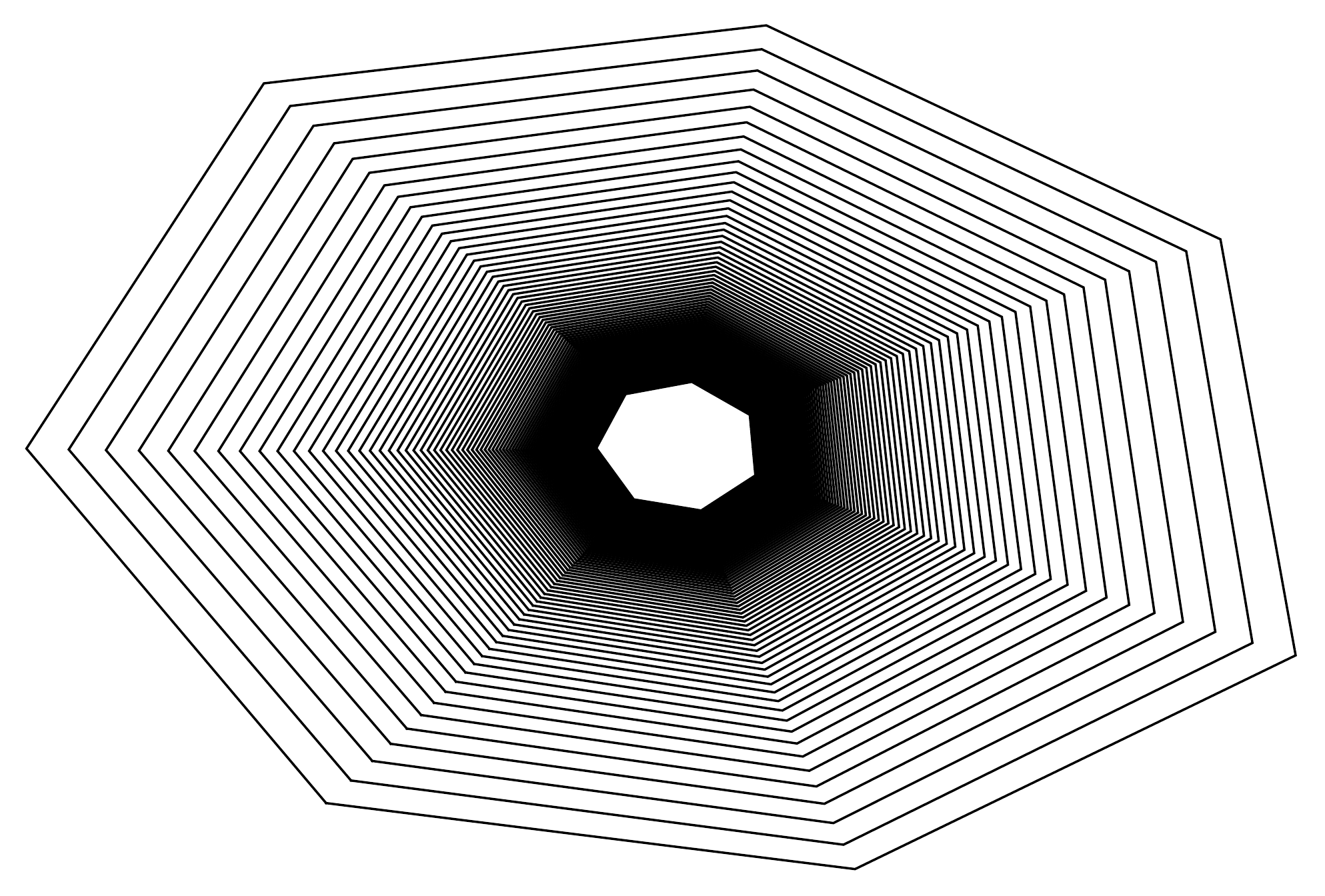}}
\hspace{0.5in}
    \subfigure[]{
    \label{hepta:c} %% label for first subfigure
    \includegraphics[width=1.5in]{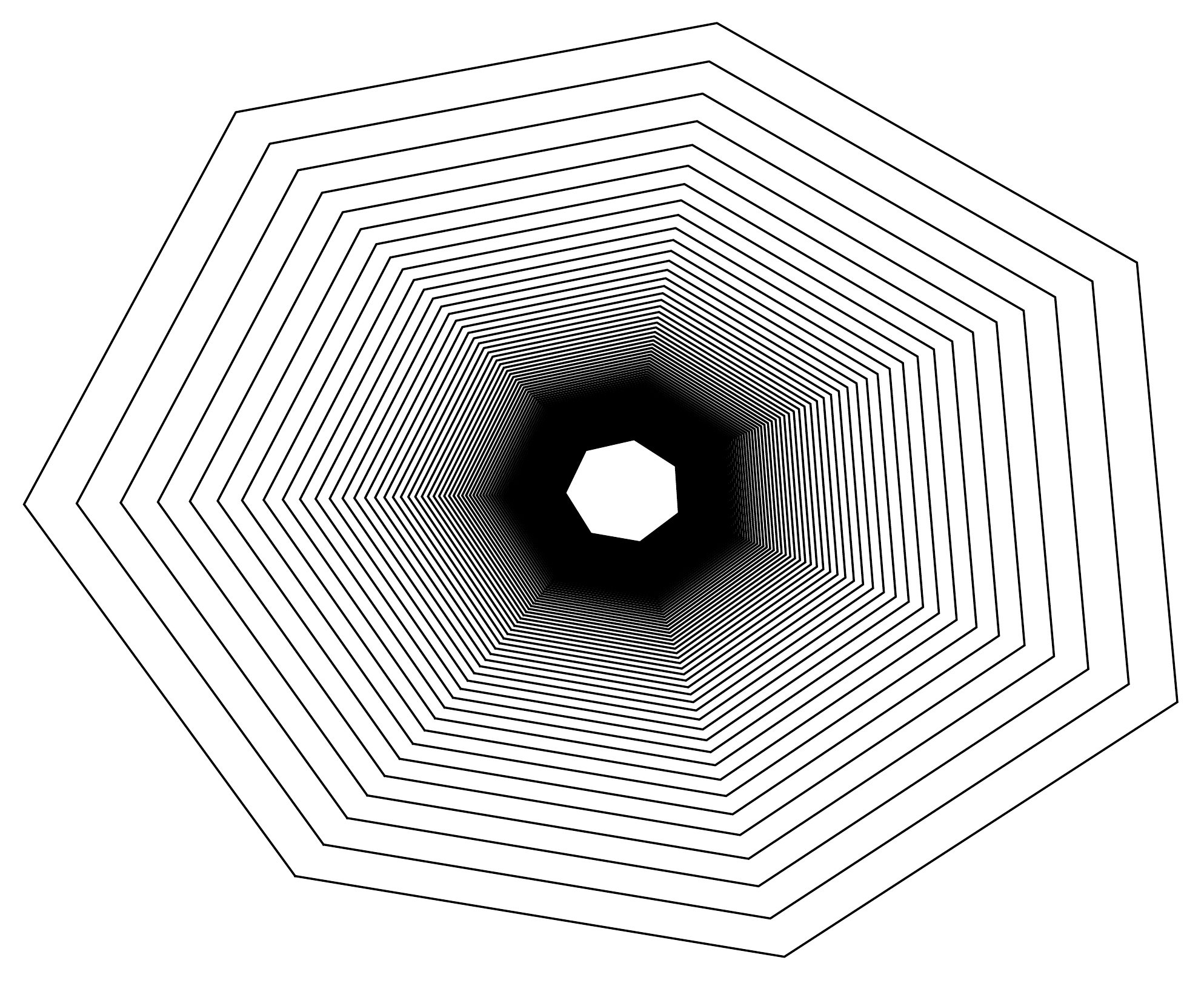}}
    %\\

   \subfigure[]{
    \label{hepta:d} %% label for first subfigure
    \includegraphics[width=1.5in]{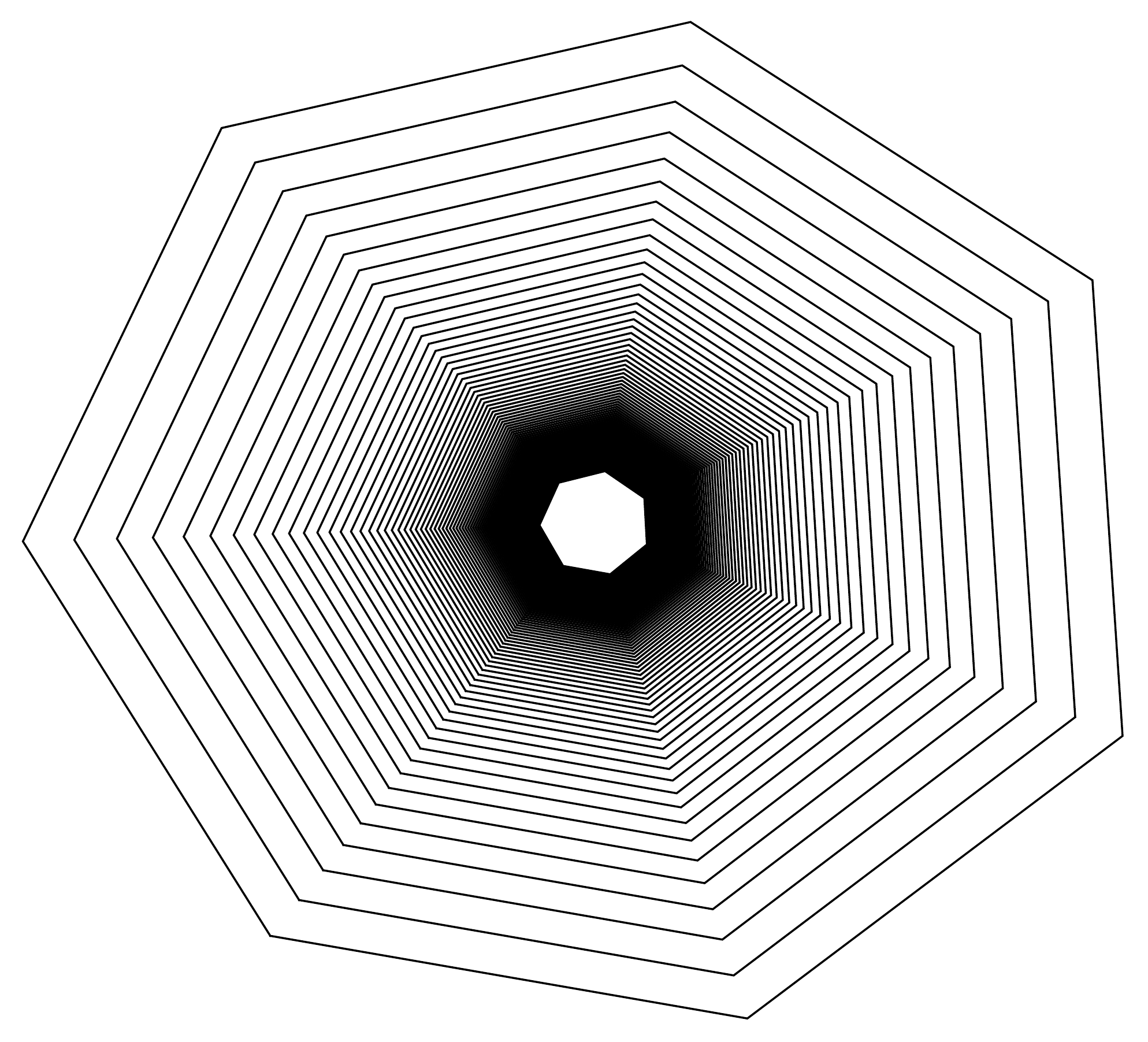}}
\hspace{0.5in}
   \subfigure[]{
    \label{hepta:e} %% label for first subfigure
    \includegraphics[width=1.5in]{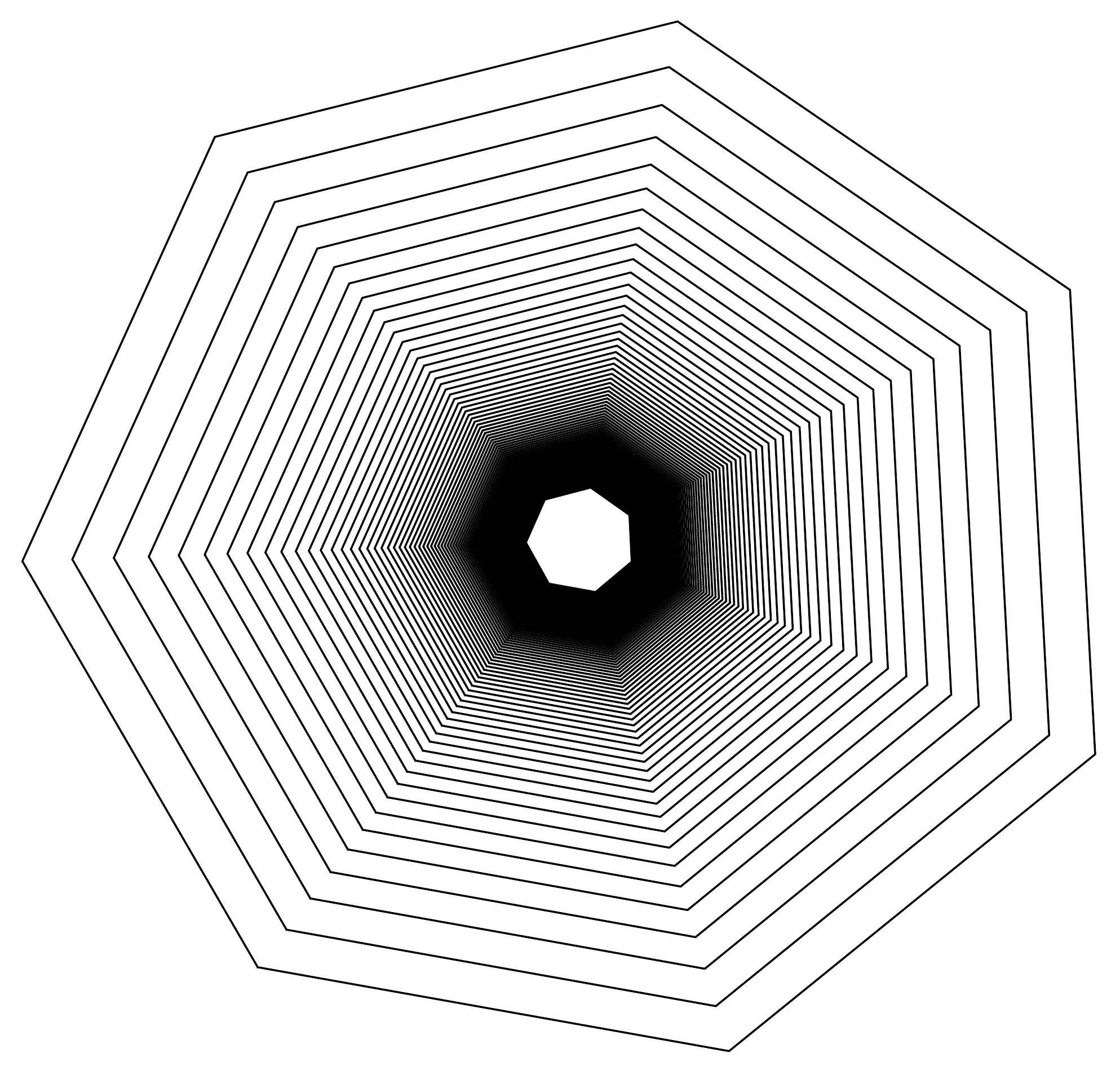}}
\hspace{0.5in}
   \subfigure[]{
    \label{hepta:f} %% label for first subfigure
    \includegraphics[width=1.5in]{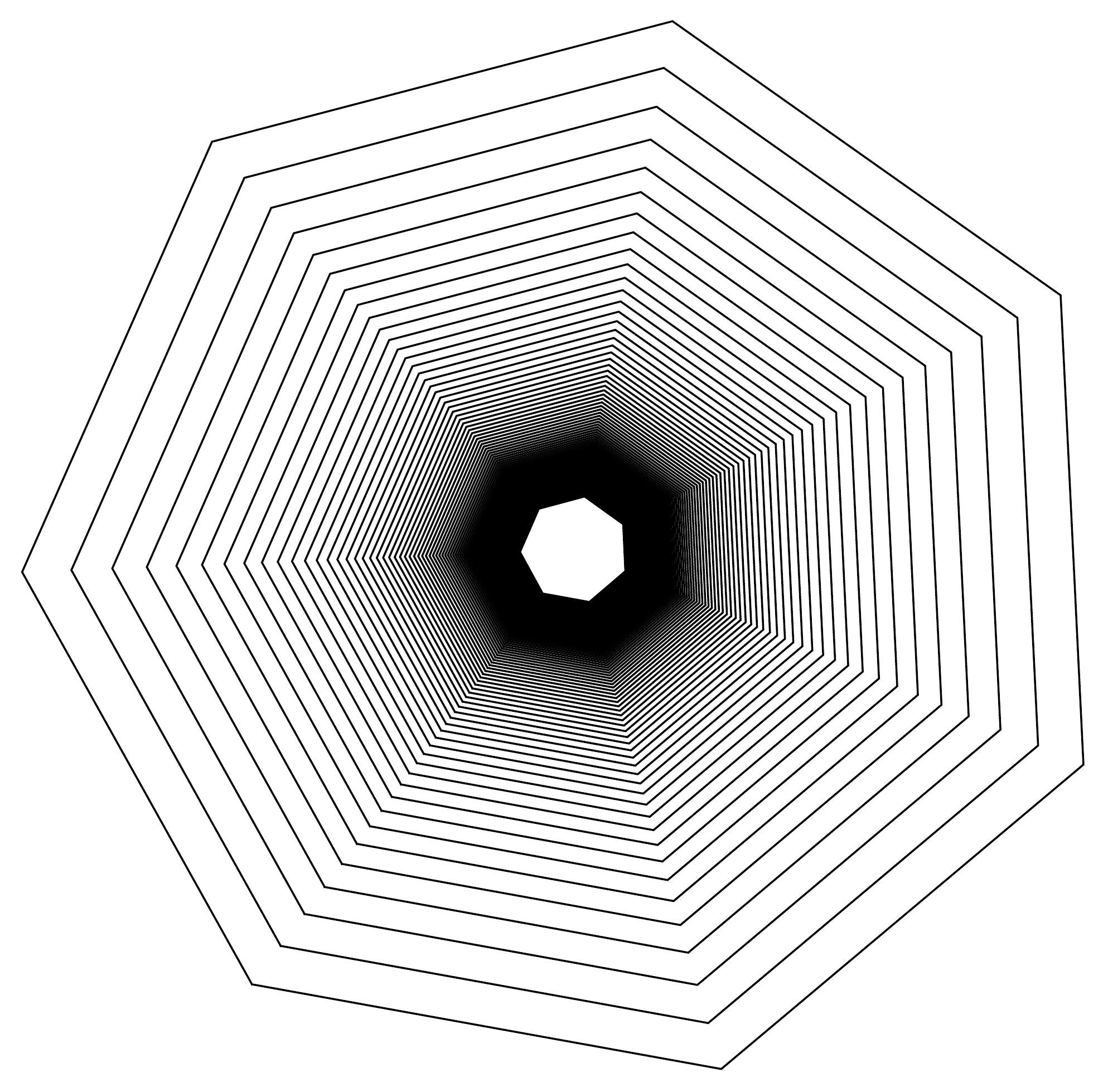}}

  \caption{Evolution of a heptagon with selected scalings in space and time.}
  \label{hepta} %% label for entire figure
\end{figure}

\begin{figure}[H]
\centering
  \subfigure[$X(1)$]{
    \label{heptaseq:a} %% label for first subfigure
    \includegraphics[width=0.8in]{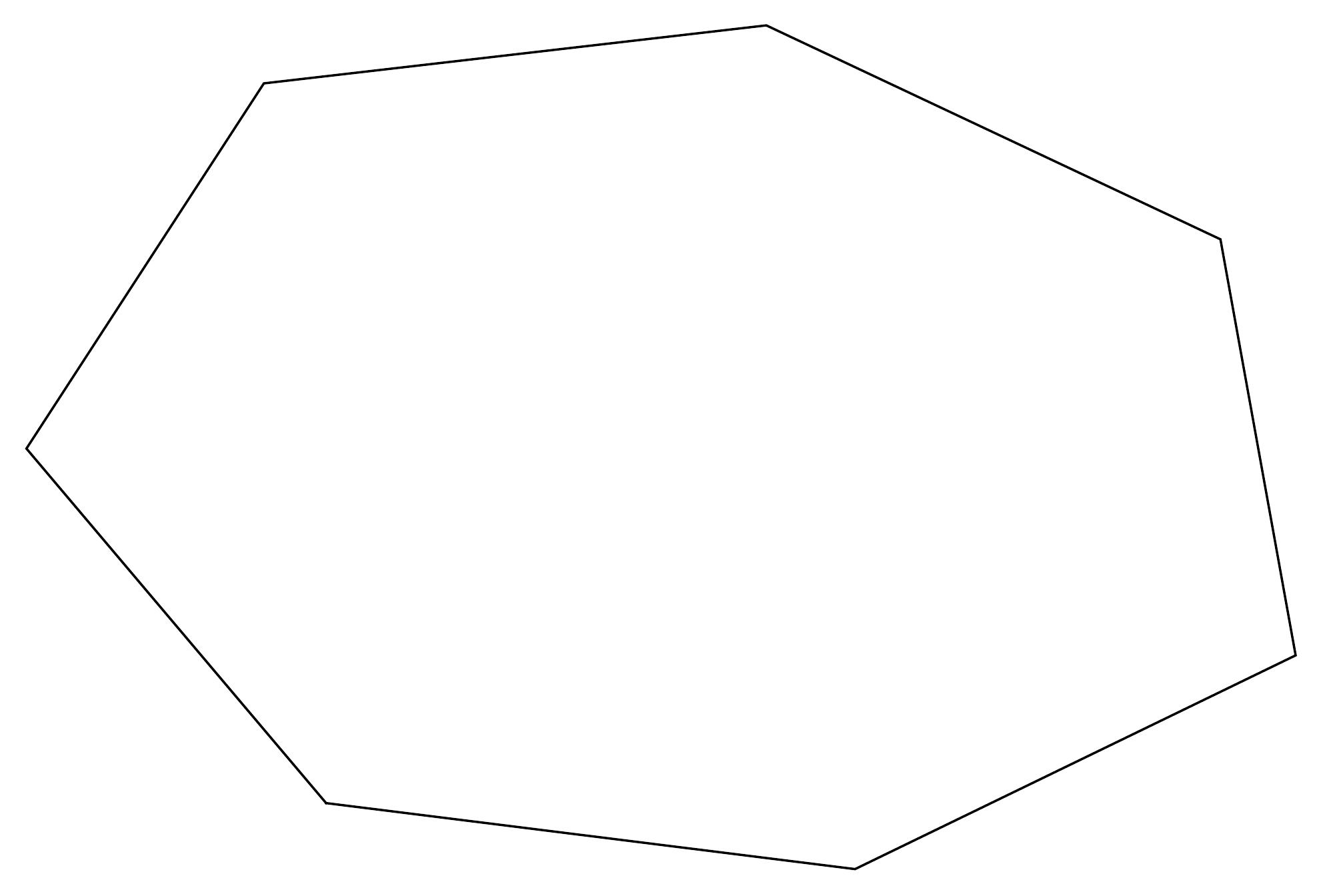}}
 \subfigure[$10X(10)$]{
    \label{heptaseq:b} %% label for second subfigure
    \includegraphics[width=0.8in]{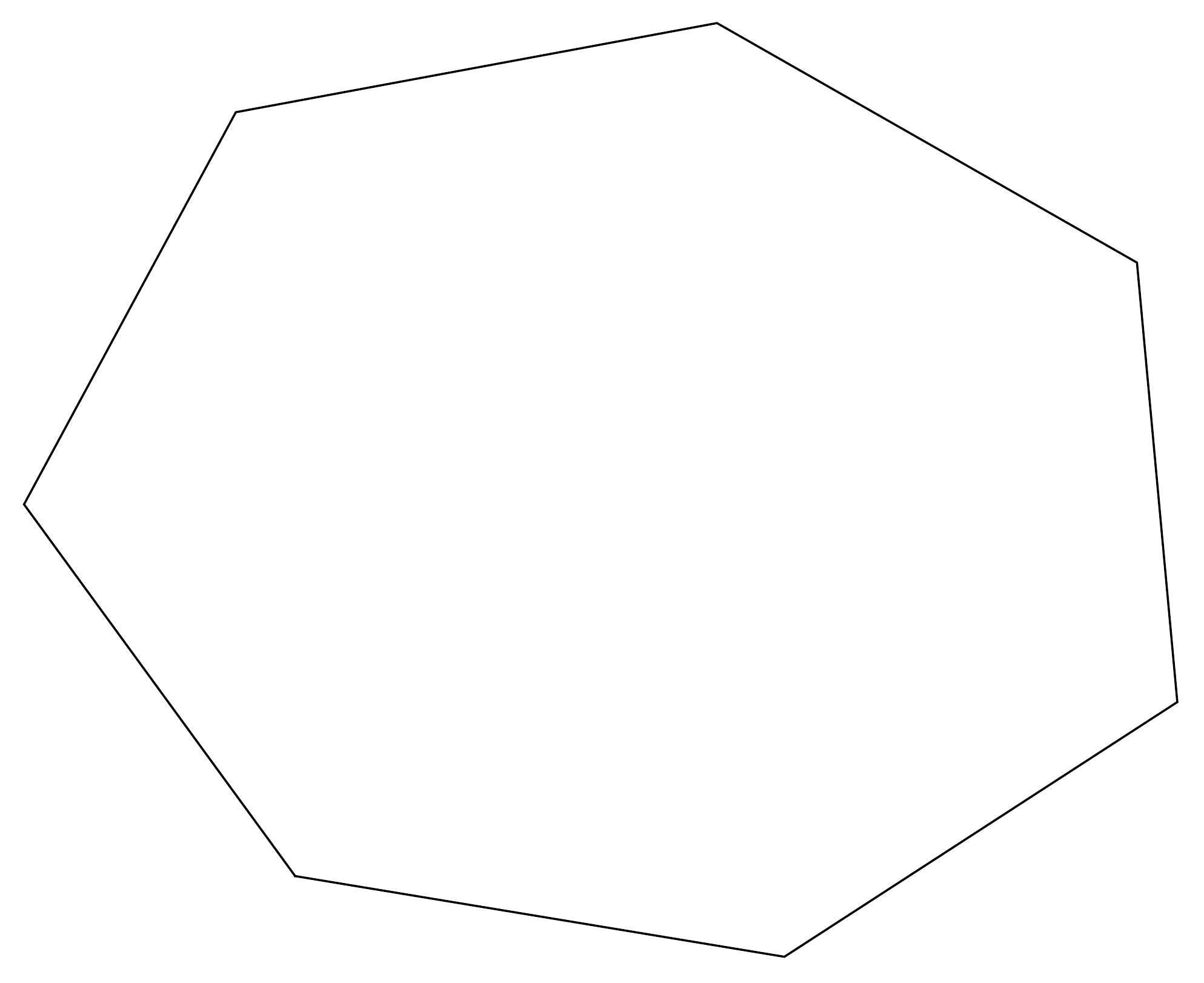}}
    \subfigure[$10^2X(10^2)$]{
    \label{heptaseq:c} %% label for first subfigure
    \includegraphics[width=0.8in]{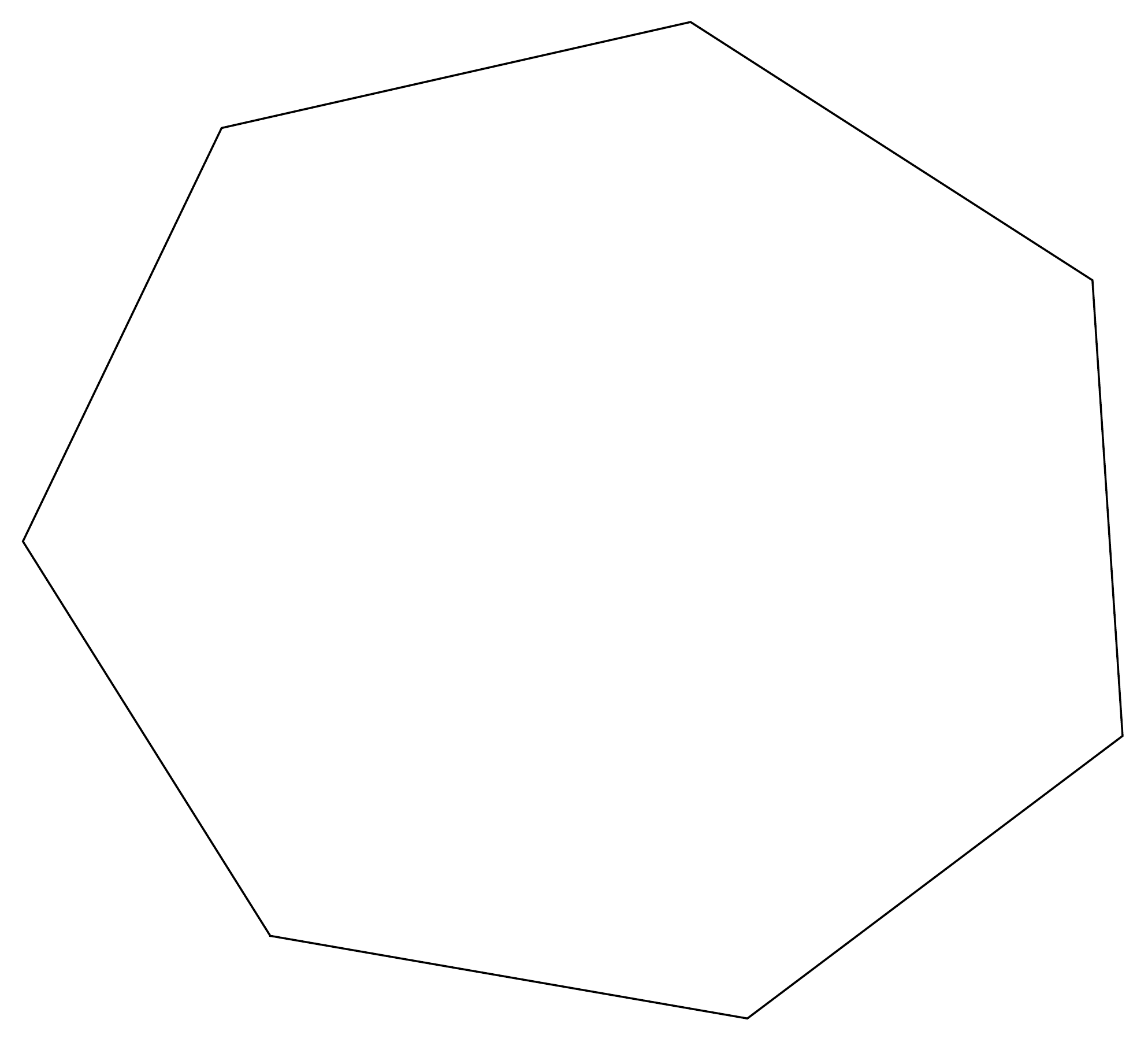}}
   \subfigure[$10^3X(10^3)$]{
    \label{heptaseq:d} %% label for first subfigure
    \includegraphics[width=0.8in]{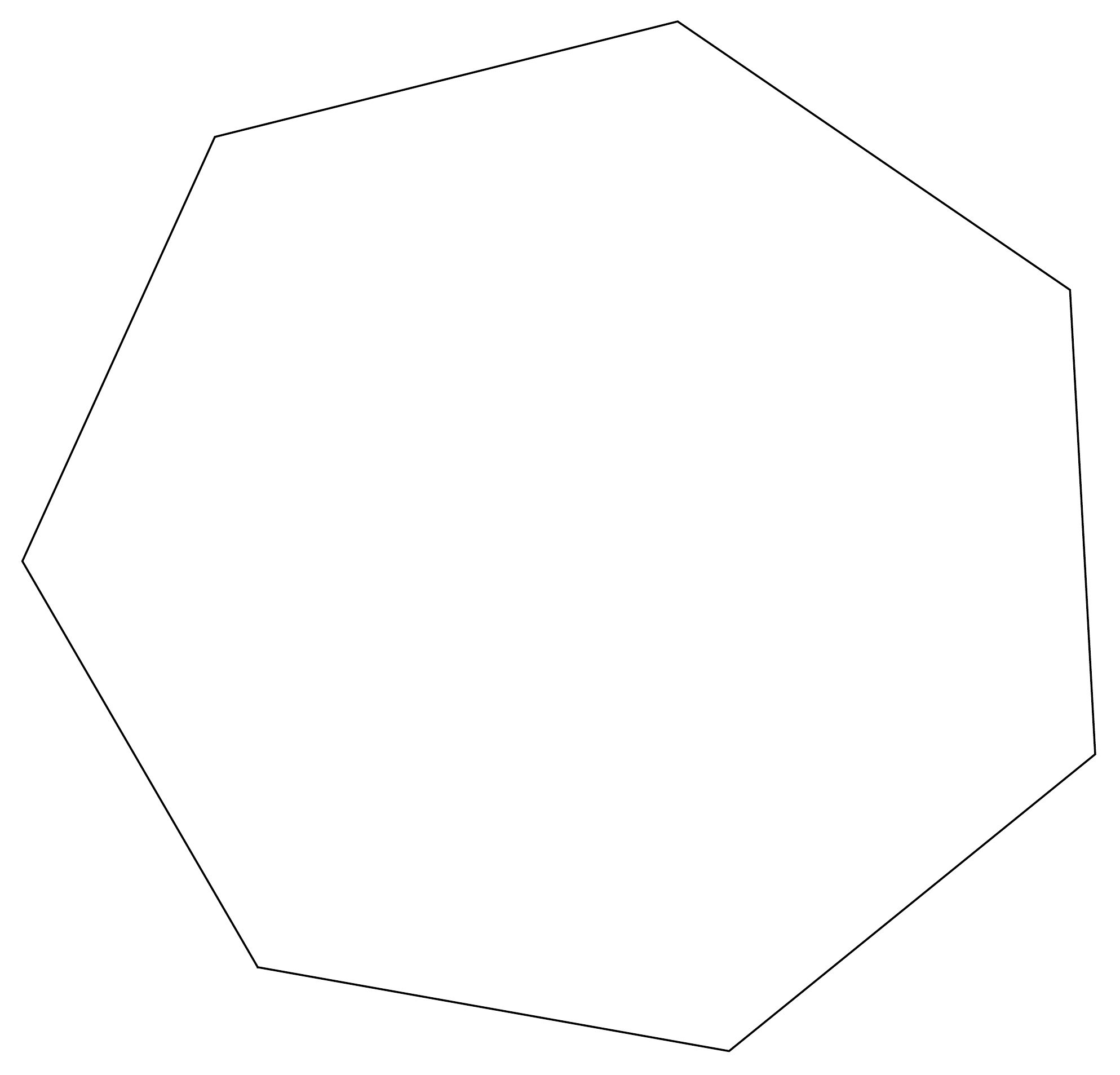}}
   \subfigure[$10^4X(10^4)$]{
    \label{heptaseq:e} %% label for first subfigure
    \includegraphics[width=0.8in]{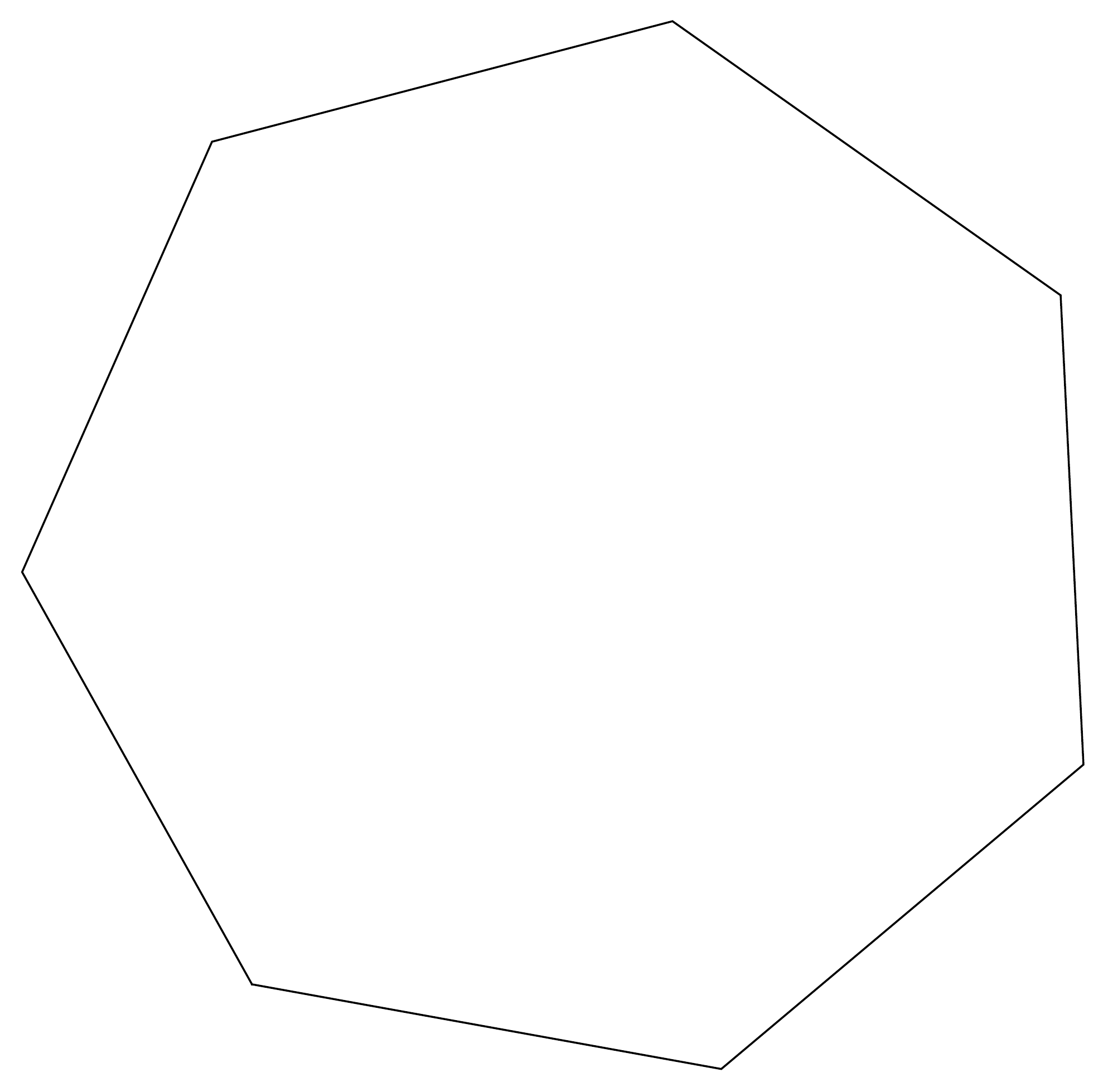}}
   \subfigure[$10^5X(10^5)$]{
    \label{heptaseq:f} %% label for first subfigure
    \includegraphics[width=0.8in]{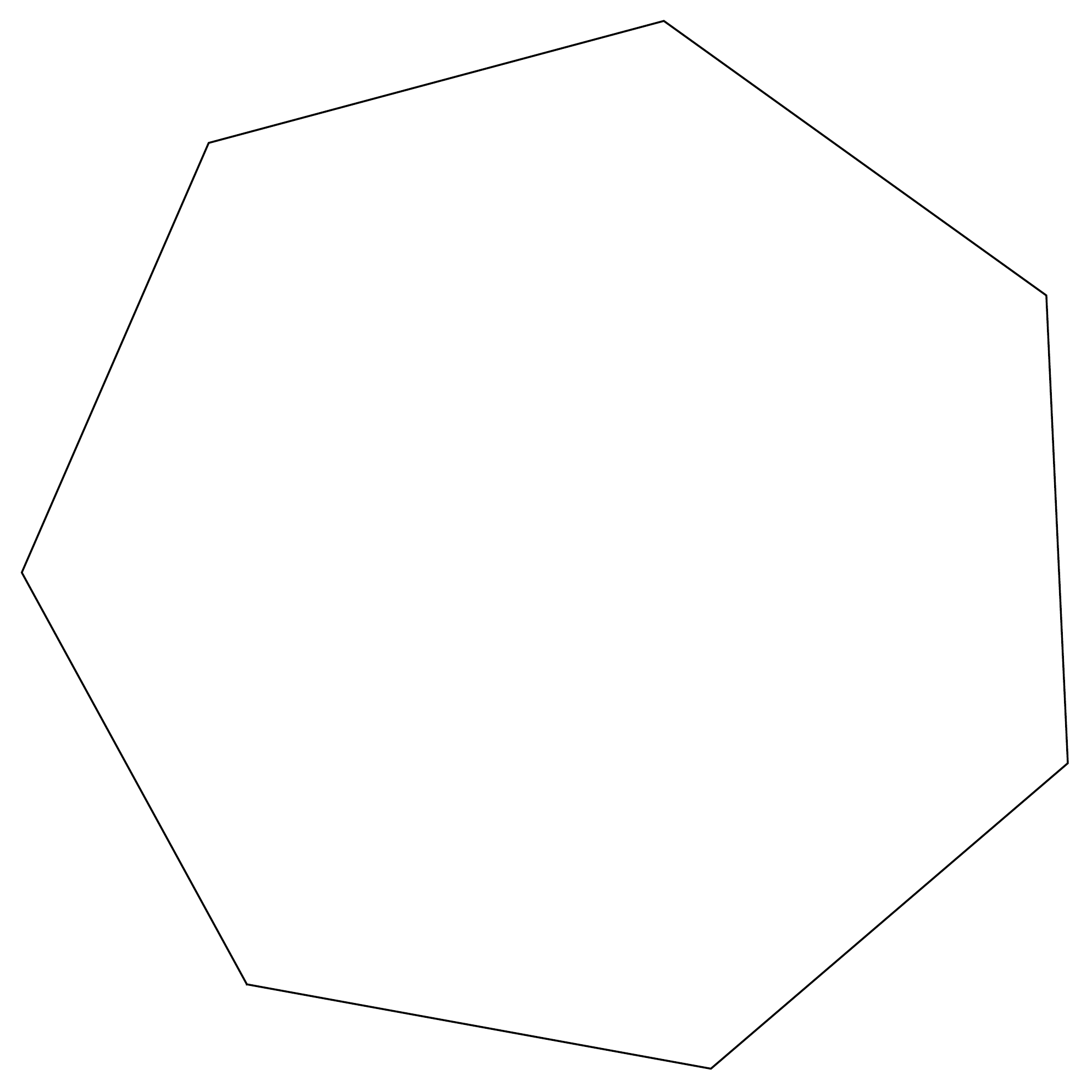}}

  \caption{The smallest heptagon in each graph of Figure \ref{hepta}.}
  \label{heptaseq} %% label for entire figure
\end{figure}
\end{example}

The next example says, our result in Theorem \ref{stableshapethm4} is sharp since it is possible
to have a locally stable rhombus.
\begin{example}\label{example:rhombus}
  In Figure \ref{quadtest}, we look at the flow starting at a rectangle and also starting at another quadrilateral. It appears that the rectangle evolves to a square while the other quadrilateral evolves
  to a rhombus that is not a square.

%   We believe that the rhombus is a locally stable shape under (\ref{flowmatrix}), however, we need re-linearize our system around the rhombus, while the derivation should be similar.
\begin{figure}
  \centering
  \includegraphics[scale=0.5]{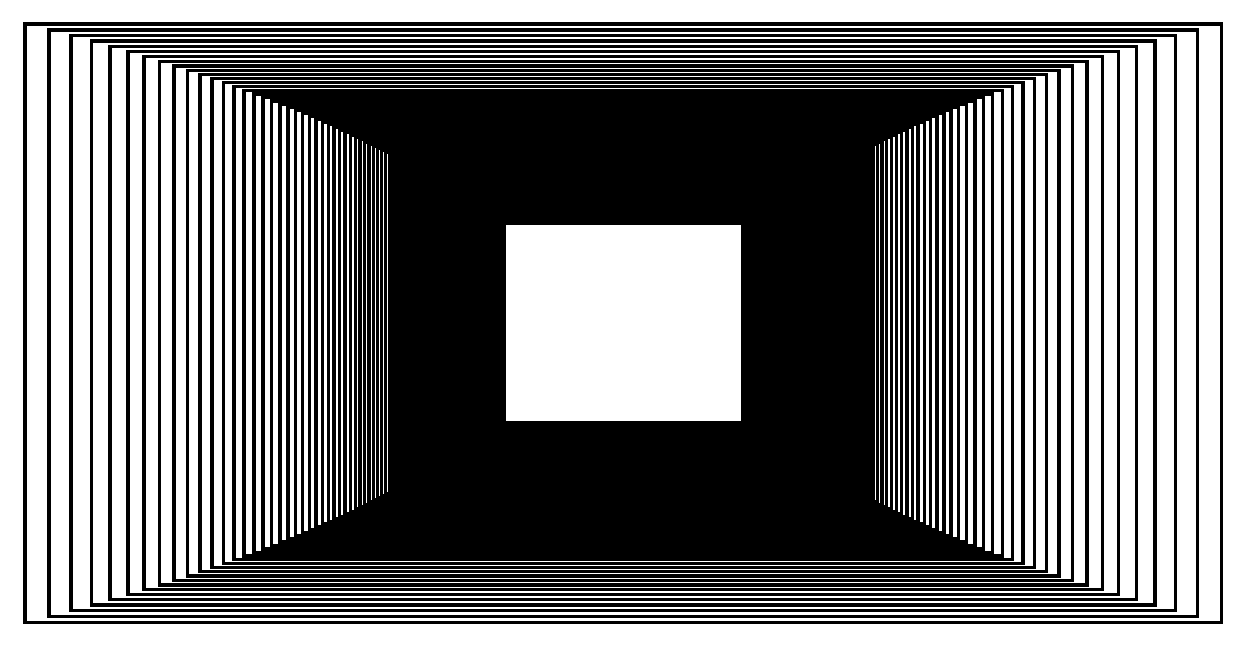}\hfill
  \includegraphics[scale=0.3]{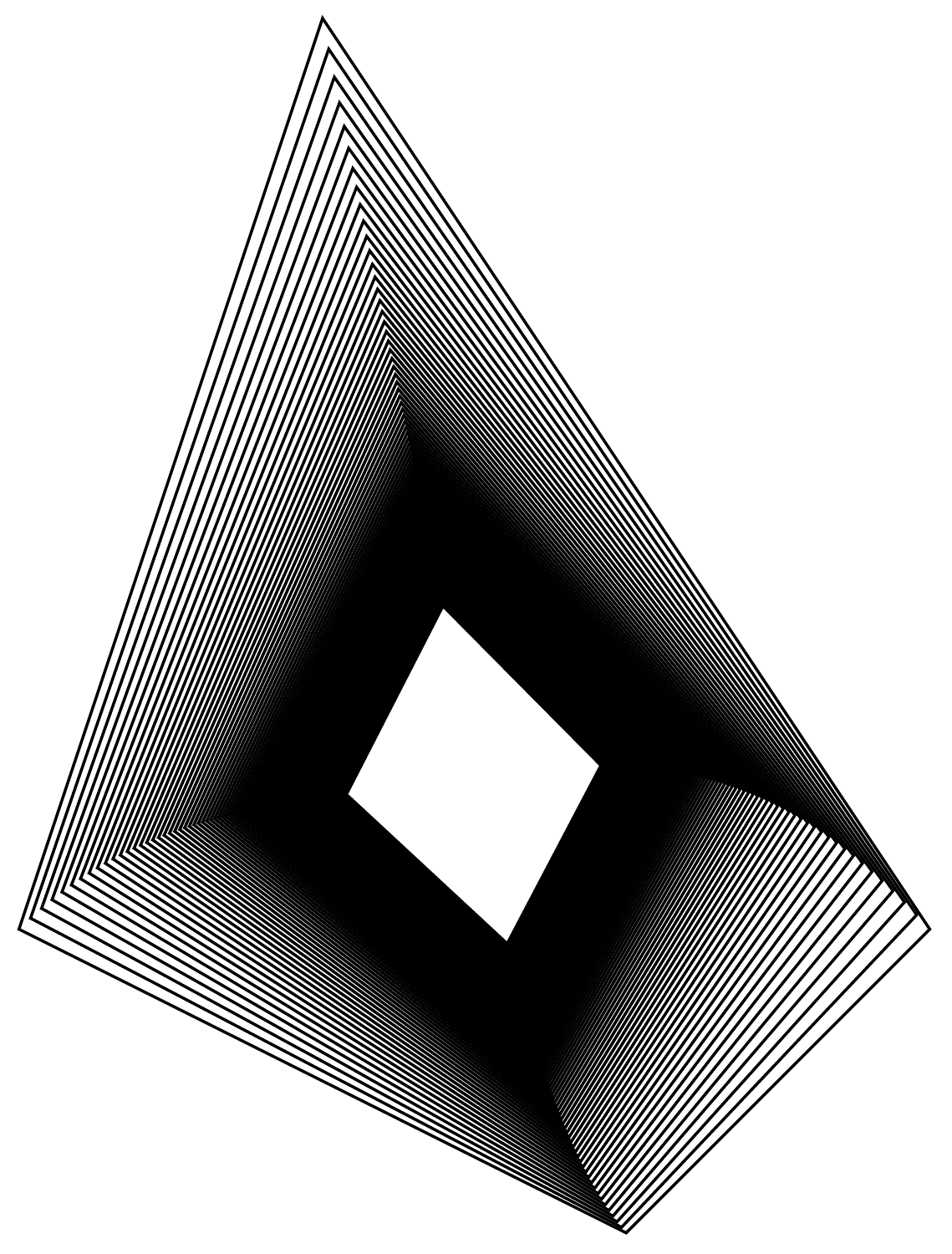}
  \caption{Evolution of the quadrilateral with $\beta=1$ under (\ref{flowmatrix}): (left) a rectangle shrinks to a square; (right) a quadrilateral shrinks to a rhombus.}\label{quadtest}
\end{figure}

\end{example}

\section{Evolution of the triangle}\label{sectria}

Inspired by the techniques used in \cite{Naka}, we obtain the following result
for the triangle, $N=3$.

\begin{theorem}\label{trilimit}
  Under the $\beta$-polygon flow, an arbitrary (nondegenerate) triangle shrinks to a point and converges to a regular triangle if appropriately rescaled.
\end{theorem}
\begin{proof}
  For a triangle, it is sufficient to show that the angles all converge to $\pi/3$.
  The possible angles for a (nondegenerate) counterclockwise oriented triangle
  form the following region $\Omega$:
  \begin{equation}\label{triregion}
    \Omega=\{(\theta_0,\theta_1,\theta_2)|\theta_0+\theta_1+\theta_2=\pi,0<\theta_0,\theta_1,\theta_2<\pi\}.
  \end{equation} The area $S$ can be expressed as
  $$S=\frac{1}{2}l_0l_2\sin\theta_0=\frac{1}{2}l_0l_1\sin\theta_1=\frac{1}{2}l_1l_2\sin\theta_2.$$
  Using this relation and (\ref{angle}), we have
  \begin{displaymath}
    \begin{split}
      \frac{d{\theta}_0}{dt} & =\frac{1}{l_0l_2}\left[(l_2^{\beta+2}+l_0^{\beta+2})\sin\theta_0-l_1^{\beta+1}l_2\sin\theta_1-l_1^{\beta+1}l_0\sin\theta_2)\right] \\
%        & =\frac{1}{l_0l_2}\left[l_1l_2\sin\theta_1(l_2^\beta-l_1^\beta)+l_0l_1\sin\theta_2(l_0^\beta-l_1^\beta)\right]\\
        &=\frac{1}{2S}\left[l_1^2\sin^2\theta_1(l_2^\beta-l_1^\beta)+l_0^2\sin^2\theta_0(l_0^\beta-l_1^\beta)\right].
    \end{split}
  \end{displaymath}
  Similarly, we have $$\frac{d{\theta}_1}{dt}=\frac{1}{2S}\left[l_2^2\sin^2\theta_2(l_0^\beta-l_2^\beta)+l_1^2\sin^2\theta_1(l_1^\beta-l_2^\beta)\right]$$
  and $$\frac{d{\theta}_2}{dt}=\frac{1}{2S}\left[l_0^2\sin^2\theta_0(l_1^\beta-l_0^\beta)+l_2^2\sin^2\theta_2(l_2^\beta-l_0^\beta)\right].$$
  Let us introduce a function $$V(\theta_0, \theta_1, \theta_2)=-(\pi-\theta_0)(\pi-\theta_1)(\pi-\theta_2).$$
  The function $V$ is negative in $\Omega$, zero on $\partial\Omega$, and has a unique minimum at $P=(\pi/3,\pi/3,\pi/3)$. Its time derivative is given by
  \begin{displaymath}
    \begin{split}
      \frac{dV}{dt} & =\frac{d{\theta}_0}{dt}(\pi-\theta_1)(\pi-\theta_2)+\frac{d{\theta}_1}{dt}(\pi-\theta_0)(\pi-\theta_2)+\frac{d{\theta}_2}{dt}(\pi-\theta_0)(\pi-\theta_1) \\
        & =\frac{1}{2S}l_1^2\sin^2\theta_1(l_2^\beta-l_1^\beta)(\theta_0-\theta_1)(\pi-\theta_2)+\frac{1}{2S}l_0^2\sin^2\theta_0(l_0^\beta-l_1^\beta)(\theta_0-\theta_2)(\pi-\theta_1)\\
        &\qquad +\frac{1}{2S}l_2^2\sin^2\theta_2(l_0^\beta-l_2^\beta)(\theta_1-\theta_2)(\pi-\theta_0).
    \end{split}
  \end{displaymath}
  The right-hand side is negative on $\Omega-P$ and zero at $P$. Thus $V$ is a Lyapunov function and, therefore, $P$ is asymptotically stable.
\end{proof}

Since the evolution of triangles is understood by Theorem \ref{trilimit}, in the rest of this paper we assume that $N \geq 4.$

\section{Self-similar solutions and the rescaled flow}\label{secselfsim}
One way to study asymptotic behavior of geometric flows is to try to show
that limiting flows converge to self-similar solutions in some sense. Self-similar solutions are special solutions
that do not change shape as they evolve. In other words, the initial data
 determines the shape of the solution. This property yields one of the benefits of finding the self-similar
 solutions: the time variable can be separated out. In \cite{abr}, Abresch and Langer
 classify all of the self-similar solution of the curve shortening flow. In \cite{ang}, Angenent shows that a convex immersed plane
 curve that evolves by its curvature will either shrink to a point in an asymptotically self-similar
 manner (as described by Abresch and Langer \cite{abr}), or else there exists a rescaled flow coverging to the graph
 of the grim reaper (a noncompact self-similar solution). Similarly, in \cite{schn}, the authors proved that there is a rescaled lens-shaped network that
 contracts smoothly to a unique self-similar solution of the planar network flow. We will show
 that solutions to the $\beta$-polygon flow converge converge asymptotically to self-similar solutions.

 Recall the definition
of point from Proposition \ref{prop:fixed points}.
\begin{definition}\label{selfsimilar}
  We say $\mathbf{X}(t)$ is a self-similar solution of (\ref{flowmatrix}) if there exists a polygon $X_0$, a scaling function $\lambda(t)$, and a point $Q$ such that $\mathbf{X}(t)=\lambda(t)X_0+Q$ satisfies (\ref{flowmatrix})
\end{definition}

By Theorem \ref{exist} we must have that $\lambda(t) \rightarrow 0$ and $Q=\displaystyle\lim_{t\rightarrow\infty}\mathbf{X}(t)$ as
$t \rightarrow \infty$.

When $\beta=0$, the system (\ref{flowmatrix}) becomes a linear system as studied in \cite{dav}, and the corresponding $N\times N$ matrix $M$ is:
\begin{equation}\label{linearM}
M=\left(
                       \begin{array}{cccccc}
                         -2 & 1 & 0 & \cdots & 0 & 1 \\
                         1 & -2 & 1  & \ddots & 0 \\
                         0 & 1 & -2 & 1 & 0 & \vdots \\
                         \vdots & 0 & \ddots & \ddots & \ddots & 0 \\
                         0 & \ddots & 0 & 1 & -2 & 1 \\
                         1 & 0 & \cdots & 0 & 1 & -2 \\
                       \end{array}
                     \right).\end{equation}
Because $M$ is a circulant matrix, $\mathbb{C}^N$ has a basis of eigenvectors consisting of $N$-th roots of unity:
\begin{equation}\label{eigvecofM}
  P_k=(1,\omega^k,\omega^{2k},\cdots,\omega^{2(n-1)k})^T\qquad (k=0,\cdots,N-1),
\end{equation}
where $\omega^{N}=1$ and ``T" signifies the transpose. We can think of these these vectors
as listing the vertices of a regular, oriented (possibly star-like) polygon in the complex plane by drawing each
entry of the vector in the complex plane and connecting consecutive entries by arrows.
The eigenvalue corresponding to the eigenvector $P_k$ can then be computed to be
\begin{equation}\label{eigofM}
  \lambda_k=-4\sin^2(\pi k/N)\qquad (k=0,\cdots,N-1).
\end{equation}

%\subsection{Turning the self-similar solution into a fixed point}\label{turnsec}
%In this section,

%By the remark \ref{invariantrigid} and the existence and uniqueness theorem \ref{exist}, \ref{unique}, we are only need to study the local behavior of the solution around $P_1.$

For $\beta>0$, we see that the regular polygons are still self-similar solutions.

\begin{lemma}\label{selfp1lem}
  The regular $N$-gons $P_k$ as defined in (\ref{eigvecofM}) are self-similar solutions of (\ref{flowmatrix}), i.e., if
  $a(t)=(1-\beta l^\beta\lambda_k t)^{-1/\beta}$, where $l$ is the edge length of $P_k$
  and $\lambda_k=-4\sin^2(\pi k/N)$ is the corresponding eigenvalue of $M$, then $P(t)=a(t)P_k$
%  	\begin{equation}\label{selfp1}
%       P(t)=a(t)P_1,\qquad a(t)=(1-\beta l^\beta\lambda_1 t)^{-1/\beta},
%    \end{equation}
    is a solution of (\ref{flowmatrix}).
\end{lemma}
\begin{proof}
  Using $P(t)=a(t)P_k$ in (\ref{flowmatrix}), we obtain
  \begin{displaymath}
    \frac{da}{dt} P_1=M_P P=a^{1+\beta}M_{P_k} P_k=a^{1+\beta}l^\beta M P_k=a^{1+\beta}l^\beta \lambda_k P_k.
  \end{displaymath}
  We then obtain $a(t)$ as the solution to the above differential equation with $a(0)=1$.
\end{proof}

We will study the asymptotic stability of the solutions of (\ref{flowmatrix}) around the regular polygon
 $P_1$. Due to the
invariance property, Lemma \ref{leminvariant}, it is sufficient to
study the local behavior near $P_1$.

Since the flow converges to a point by Theorem \ref{exist}, in order to study local behavior, we need to find
an appropriate rescaling. Let $\alpha(t):\mathbb{R}\rightarrow \mathbb{R}^+$ be some positive scaling function
and $X(t)$ be a solution of (\ref{flowmatrix}). Let $\bar{X}$ denote the vector of all ones multiplied by the average
$\frac 1 N \sum_{j=0}^{N-1} X_i$. Using Proposition \ref{prop:M properties}, we have
$$\frac{d}{dt}(\alpha (X-\bar{X}))=\frac{d\alpha}{dt}(X-\bar{X})+\alpha M_X X=\frac{d\alpha}{dt}\frac{1}{\alpha}(\alpha (X-\bar{X}))+\frac{1}{\alpha^\beta}M_{\alpha (X-\bar{X})} (\alpha (X-\bar{X})).$$

Letting $Y=\alpha (X-\bar{X})$,
  we obtain the following nonlinear system:
\begin{equation}
  \frac{d}{dt}Y=\frac{d\alpha}{dt}\frac{1}{\alpha}Y+\frac{1}{\alpha^\beta}M_{Y}Y.
\end{equation}
By letting $\alpha(t)=1/a(t)$, where $a$ is the function in Lemma \ref{selfp1lem}, we have
\begin{equation}\label{turnself}
  \frac{d}{dt}Y=-a^\beta l^\beta \lambda_1 Y+a^\beta M_Y Y,
\end{equation}
where $a$, $l$, and $\lambda_1$ are the same as in Lemma \ref{selfp1lem}. It is clear that the regular $N$-gon $P_1$
is an equilibrium point of this system and $c(t)P_1$ is a self-similar solution with $c(t)\rightarrow1$ as $t\rightarrow\infty.$

%Notice that this is not an autonomous system. 
Consider a new time variable $\tau$ determined by $$\frac{dt}{d\tau}=\frac{1}{l^\beta a^\beta},$$  giving
 $$\tau=\frac{\ln(1-\beta l^\beta\lambda_1t)}{-\beta\lambda_1}.$$ Converting the system to
 a function of $\tau$ results in the following equation:
\begin{equation}\label{turnselfauto}
  \frac{dY}{d\tau}=-\lambda_1 Y+\frac{1}{l^\beta}M_Y Y.
\end{equation}
This leads to the following definition.
\begin{definition}
We call the flow in (\ref{turnselfauto}) the $\lambda_1$-\emph{rescaled $\beta$-polygon flow}.
\end{definition}

%\begin{remark}
%  From the construction, we know that it is a one to one correspondence between the solution of (\ref{flowmatrix}) and the solution of (\ref{turnselfauto}). Indeed, $Y(\tau)$ is a solution of (\ref{turnselfauto}) if and only if $a(t)Y(\tau(t))$ is a solution of (\ref{flowmatrix}). Therefore, the uniqueness Theorem \ref{unique} implies the unique solution of (\ref{turnselfauto}). Moreover, since the corresponding solutions only differ by a scaling, this allows us to use the system (\ref{turnselfauto}) to explore the shape of the solution of (\ref{flowmatrix}). Note that the regular polygon $P_1$ is a equilibrium point of (\ref{turnselfauto}), in the next section, we will linearize the system (\ref{turnselfauto}) at $P_1$ to study the local stability at any regular polygon.
%\end{remark}

%To investigate the asymptotic behavior of the solution around the regular $N$-gon $P_1$ under (\ref{flowmatrix}), it's enough to study it under the  system (\ref{turnselfauto}). Since $t$ defines for all $t>0$, so $\tau$ is. As we see above, regular $N$-gon $P_1$ is a equilibrium point of (\ref{turnselfauto}), to look at the asymptotic behavior of the solution around it, we need linearize the system (\ref{turnselfauto}).

Motivated by Huisken's monotonicity formula \cite{hui}, we will prove a general monotonicity formula for polygons evolving under
%(\ref{flowmatrix})
the $\beta$-polygon flow. We then use the standard blow up argument in geometric flows (see \cite{schn} and \cite{hui} for instance) to show that evolutions satisfying the bound (\ref{anglelowerbd}) are asymptotically self-similar.

Let $X(t)=\lambda(t)X_0+Q$ (see Definition \ref{selfsimilar}) be a self-similar solution of
%(\ref{flowmatrix})
the $\beta$-polygon flow with $\lambda(0)=1.$ Since $\displaystyle\frac{d}{dt}\left(\frac{X-Q}{\lambda}\right)=0,$ we have
$$-\frac{d\lambda}{dt}\frac{1}{\lambda^2}(X-Q)+\frac{1}{\lambda}M_X X=0.$$
 By letting $t=0$, we obtain the following equation that determines the self-similar
 %shrinking (i.e., $\lambda'(0)<0$)
 solution:
\begin{equation}\label{similarchar}
  M_{X_0}X_0=\frac{d\lambda}{dt}(0)(X_0-Q).
\end{equation}
%We don't worry about the existence of (\ref{similarchar}), since at least, all of the base elements (\ref{eigvecofM}) are the solutions.

Given $x_0\in\mathbb{C}^N$ and $X(t)$ a solution of the $\beta$-polygon flow, define $\rho_{x_0}(X,t)$ be the following entropy functional
\begin{equation}\label{enekernel}
  \rho_{x_0}(X,t)=\exp\left[-t^{2/\beta}\left|\strut X(t)-x_0\right|^2-\int_{0}^{t}\frac{\beta}{2}s^{2/\beta+1}\left|M_{X(s)} X(s)\right|^2ds\right],
\end{equation} where $|X|$ denotes the 2-norm of $X$ as in (\ref{lpnrom}).
We have the following monontonicity formula.

\begin{theorem}\label{monoformula}
  Let $X(t)$ be a the solution of %(\ref{flowmatrix})
  the $\beta$-polygon flow,
  %satisfying $x(t)\rightarrow x_0$ as $t\rightarrow\infty$,
  then for any polygon $x_0$ we have the formula
  \begin{displaymath}
    \frac{d}{dt}\rho_{x_0}(X,t)=-\frac{2}{\beta}\rho_{x_0}(X,t)t^{2/\beta-1}\left\vert X-x_0+\frac{\beta}{2}tM_X X\right\vert^2.
  \end{displaymath}
  \end{theorem}
  \begin{proof}
    We calculate directly from (\ref{flowmatrix}) and find
    \begin{displaymath}
      \begin{split}
        \frac{d}{dt}\rho_{x_0}(X,t) & =\rho_{x_0}(X,t)\left[-\frac{2}{\beta}t^{2/\beta-1}\left| \strut X-x_0 \right|^2-2t^{2/\beta} (X-x_0)\cdot M_X X-\frac{\beta}{2}t^{2/\beta+1}\left|\strut M_X X \right|^2\right] \\
          & =-\frac{2}{\beta}\rho_{x_0}(X,t)t^{2/\beta-1}\left\vert X-x_0+\frac{\beta}{2}t M_X X\right\vert^2.
      \end{split}
    \end{displaymath} 
  \end{proof}

Consider a polygon $X(t)$ for $t\in[0,\infty)$ that contracts to a point $x_0$ as $t\rightarrow\infty.$ Consider a sequence of positive numbers $c_k\nearrow\infty$. We rescale the polygon under a sequence of dilations given by
\begin{equation}\label{rescale}
  Y^k(\tau)=c_k\left(X(c_k^\beta\tau)-x_0\right).
\end{equation}
Note that for each $k$, $Y^k(\tau)$ converges to the origin as $\tau \to \infty$. To ensure the compactness of the family of rescaled polygons $Y^k(\tau)$, we assume that the angles $\theta_i$ of $X(t)$ satisfy a lower bound, i.e., there exist $T_0$ and $\delta>0$ such that
\begin{equation}\label{anglelowerbd}
  \inf_{t\in[T_0,\infty)}\min_{i=0,\cdots,N-1}\sin^2\theta_i\geq \delta.
\end{equation}

\begin{lemma}\label{enerybd}
  If (\ref{anglelowerbd}) is valid, then the energy of the rescaled polygons $\{Y^k\}_{k=1}^\infty$
  is uniformly bounded on the compact interval $[\epsilon,1/\epsilon]$ for any $\epsilon>0$.
  Explicitly, there exists some uniform constants $\eta(\epsilon,\alpha,N)$ and
  $\tilde{\eta}(\epsilon,\alpha,N)$ such that
  \begin{displaymath}
    \eta(\epsilon,\alpha,N)\leq F_\alpha(Y^k(\tau))\leq\tilde{\eta}(\epsilon,\alpha,N),
    \qquad\forall k\geq 0,\tau\in[\epsilon,1/\epsilon].
  \end{displaymath}
\end{lemma}
\begin{proof}
 Let $X=(X_0,\cdots,X_{N-1})\in\mathbb{C}^N$. Recall that $\theta_j$ denotes the angle at $X_j$ and $l_j$ denotes the edge-length between $X_j$ and $X_{j+1}$. By (\ref{energey}) and (\ref{flowmatrix}), we have
  \begin{displaymath}
    \begin{split}
      \frac{d}{dt}F_\alpha(X) & = \frac{1}{\alpha}\frac{d}{dt}\sum_{j=0}^{N-1}l_j^\alpha \\
        & =-\sum_{j=0}^{N-1}\left\vert l_j^\beta(X_{j+1}-X_j)+l_{j-1}^\beta(X_{j-1}-X_j)\right\vert^2\\
        & =-\sum_{j=0}^{N-1}(l_j^{2\beta+2}+\l_{j-1}^{2\beta+2}+2l_j^{\beta+1}l_{j-1}^{\beta+1}\cos\theta_j)\\
        & =-\sum_{j=0}^{N-1}\left(l_j^{2\beta+2}\sin^2\theta_j+\left\vert
         l_j^{\beta+1}\cos\theta_j+l_{j-1}^{\beta+1}\right\vert^2\right).
    \end{split}
  \end{displaymath}
  By assumption (\ref{anglelowerbd}), there exist $T_0$ such that if $t>T_0$, we have
  $$-4\sum_{j=0}^{N-1}l_j^{2\beta+2}\leq\frac{d}{dt}F_\alpha\leq-\delta\sum_{j=0}^{N-1}l_j^{2\beta+2},$$
  where $\delta$ is the lower bound in (\ref{anglelowerbd}). In finite dimensional vector spaces,
  all norms are equivalent, so there exist positive numbers $\eta_1$ and $\eta_2$ which only depend on
  $\beta$ and $N$ such that
  \begin{displaymath}
    -\eta_1 F_\alpha^{\frac{2\beta+2}{\beta+2}}\leq \frac{d}{dt}F_\alpha\leq -\eta_2 F_\alpha^{\frac{2\beta+2}{\beta+2}}.
  \end{displaymath}
  These formulas integrate to estimates of the energy of $X$; explicitly, there exist positive numbers $\mu_1,\mu_2,\mu_3,\mu_4$ such that
  \begin{equation}\label{energybdeqn}
    \left[\frac{1}{\mu_1 t+\mu_2}\right]^{(\beta+2)/\beta}\leq F_\alpha(X)\leq \left[\frac{1}{\mu_3 t+\mu_4}\right]^{(\beta+2)/\beta}
  \end{equation}
  if $t>T_0$, where $\mu_i$ for $i=1,2,3,4$ only depend on $\alpha$ and $N$.

  Now consider the rescaled polygons $Y^k(\tau)=c_k[X(c_k^\beta\tau)-x_0]$. Since $F_\alpha(Y^k(\tau))=c_k^{\beta+2}F_\alpha(X(c_k^\beta\tau))$ and since for any sequence $c_k\nearrow\infty$,
  eventually $c_k\epsilon$ must be greater than $T_0$, it follows
  from (\ref{energybdeqn}) that for $k$ large enough,
   \begin{equation}\label{rescaleenergybdeqn}
    c_k^{\beta+2}\left[\frac{1}{\mu_1 c_k^\beta\tau+\mu_2}\right]^{(\beta+2)/\beta}\leq F_\alpha(Y^k(\tau))\leq c_k^{\beta+2} \left[\frac{1}{\mu_3 c_k^\beta\tau+\mu_4}\right]^{(\beta+2)/\beta}.
  \end{equation} This completes the proof.
\end{proof}

\begin{lemma}\label{rescalecompact}
  If (\ref{anglelowerbd}) is satisfied and if $\tau$ is restricted to a compact interval $[\epsilon,1/\epsilon]$, then the rescaled polygons $\{Y^k(\tau)\}_{k=1}^\infty$ are
  contained in a compact set .
\end{lemma}
\begin{proof}
  By Lemma \ref{enerybd}, there exists a uniform bound for the energy of the rescaled polygons. Therefore, for each $k$,  we can pick a sequence $q_k\in\mathbb{C}^2$ and a uniform radius $R>0$ such that $Y^k(\tau)\subseteq B(q_k,R)$, where $B(q_k,R):=\{z\in\mathbb{C}^2:\Vert z-q_k\Vert_\alpha\leq R\}$ denotes the $\alpha$-norm ball centered at $q_k$ with radius $R$ and  $\tau\in[\epsilon,1/\epsilon]$. If $\{q_k\}_{k=1}^\infty$ is contained in a compact set, then the statement follows. Otherwise, there exist a subsequence $q_k$, still denoted by $q_k$, that tends to infinity. For $k$ sufficiently large, say $k_0$, we have the origin is not contained in $B(q_{k_0},R)$. However, Lemma \ref{bdlem} says $Y^{k_0}(t)=c_{k_0}\left[X(c_{k_0}^\beta t)-x_0\right]$ stays in the ball $B(q_{k_0},R)$ for all $t>\tau,$ which contradicts the fact that $Y^{k_0}(t)$ converges to the origin as $t\rightarrow\infty.$
\end{proof}
In fact, we have a stronger statement about the flows.
\begin{lemma}\label{lem:compactfamily}
  If (\ref{anglelowerbd}) is satisfied then for any $\epsilon >0$, the set $\{Y^k(\tau)\}_{k=1}^\infty$ of rescaled
  polygons (considered as functions of $\tau$) is contained in a compact subset of $C^0([\epsilon,1/\epsilon],\mathbb{C}^N)$.
\end{lemma}
\begin{proof}
By Lemma \ref{rescalecompact}, the polygons $Y^k(\tau)$ are pointwise bounded for each $\tau$. Since this implies a uniform
bound on the edge lengths and since $\beta>0$, we see that there is a bound on the derivative $dY^k/d\tau$, uniformly
in $\tau$ and $k$. The lemma follows from the Arzela-Ascoli Theorem.
\end{proof}
We are now able to prove Theorem \ref{contractsimilar}, that nonsingular solutions converge to shrinking self-similar
solutions.

\begin{proof}[Proof of Theorem \ref{contractsimilar}]

Theorem \ref{monoformula} says $\rho_{x_0}(X,t)$ is monotonically decreasing in time and bounded below. Therefore, the limit
\begin{displaymath}
  \rho_{x_0}(X,\infty)=\lim_{t\rightarrow\infty}\rho_{x_0}(X,t)
\end{displaymath} exists and is finite. Moreover, it satisfies
\begin{displaymath}
  \rho_{x_0}(X,\infty)-\rho_{x_0}(X,t)=-\frac{2}{\beta}\int_{t}^{\infty}\rho_{x_0}(X,s)s^{2/\beta-1}
  \left|X(s)-x_0+\frac{\beta}{2}s M_{X(s)} X(s) \right|^2 ds.
\end{displaymath}
Changing variables according to the rescaling described by (\ref{rescale}), we obtain
\begin{equation}\label{rescalemono}
  \rho_{x_0}(X,\infty)-\rho_{x_0}(X,t)=-\frac{2}{\beta}\int_{t/c_k^\beta}^{\infty}
  \rho_{0}(Y^k,\tau)\tau^{2/\beta-1}\left|Y^k(\tau)+\frac{\beta}{2}\tau M_{Y^k(\tau)} Y^k(\tau)\right|^2d\tau.
\end{equation}

%????
% By Lemma \ref{rescalecompact}, for a fixed $\tau$ the sequence $Y^k_\tau$ has a convengent subsequence, which we will
% continue to denote by $Y^k_\tau$, that converges to a limit polygon $Y_\tau$ and satisfies the estimate in (\ref{rescaleenergybdeqn}).
%  By passing $k$ to infinity, and moreover, $Y(\tau)$ is continuous because of the uniform bound on $\{dY_\tau^k/d\tau\}_{k>0,\tau\in[\epsilon,1/\epsilon]}$ for any $\epsilon>0$.

%Now for each $c_k$ we look at the flow $Y^k_\tau$ at time $\tau$. This defines a sequence $t_k=c_k^\beta\tau\rightarrow\infty$.

Fix $\epsilon >0$. By Lemma (\ref{lem:compactfamily}), we can extract a subsequence, which we still denote by $Y^k$ that converges to a limit polygon  $Y$ and satisfies the estimate in (\ref{rescaleenergybdeqn}).

Applying (\ref{rescalemono}) with $t=t_k$ chosen so that $t=\epsilon c_k^\beta$, we find that
\begin{displaymath}
  \frac{2}{\beta}\int_{\epsilon}^{\infty}\rho_{0}(Y^k,\tau)\tau^{2/\beta-1}\left|Y^k(\tau)+\frac{\beta}{2}\tau M_{Y^k(\tau)} Y^k(\tau)\right|^2d\tau\rightarrow0,
\end{displaymath} as $k\rightarrow\infty.$
Thus $Y$ satisfies
\begin{displaymath}
  Y(\tau)+\frac{\beta}{2}\tau M_{Y(\tau)} Y(\tau)=0,
\end{displaymath}
 when $\tau>\epsilon$ and $Y$ is a self-similar solution. Since we can do this for each $\epsilon> 0$, the result follows.

%Since the limiting flow does not depend on the sequence $(c_k)_{k=1}^\infty$ chosen, we obtain the stated result by choosing an appropriate sequence $(c_k)_{k=1}^\infty.$
\end{proof}

\section{Small perturbations around the regular polygon}\label{seclocal}

In the previous section we saw how to turn a self-similar solution (regular $N$-gon) into an equilibrium point of the $\lambda_1$-rescaled $\beta$-polygon flow (\ref{turnselfauto}). In this section, by linearizing the rescaled system and using the center manifold theorem, we shall prove Theorems \ref{stableshapethm5} and \ref{stableshapethm4}.

%\subsection{Main result}
%
%??Explain the following asymptotic stability definition either in section 1 or here??
%
%%\begin{theorem}\label{stableshapethm5}
%%  Assume $N\geq5$. Under the flow (\ref{flowmatrix}), for any regular $N$-gon, the regular shape is
%%  asymptotically stable under arbitrarily small perturbations.
%%\end{theorem}
%
%??Explain the outline including linearization scheme since it is not just a straight linearization of the flow??

\subsection{Linearization}
In order to properly describe the linearization, we will consider the flow with real coordinates.
Given $Y\in\mathbb{C}^N$, we identify $\mathbb{C}$ with $\mathbb{R}^2$ and write $Y=(Y^r,Y^i)^T,$ where $Y^r=(y_0^r,...,y^r_{N-1})$ and $Y^i=(y_0^i,\ldots,y^i_{N-1})$ give the vectors of real and imaginary parts
respectively. We can now write the $\lambda_1$-rescaled $\beta$-polygon flow as a $2N$-dimensional system:
\begin{equation}\label{realselfauto}
  \frac{dY}{d\tau}=-\lambda_1 Y+\frac{1}{l^\beta}\begin{pmatrix}
                       M_Y & 0 \\
                      0 & M_Y \\
                     \end{pmatrix} Y=:-\lambda_1 Y+\frac{1}{l^\beta}F(Y),
\end{equation} where $M_Y$ is defined in (\ref{flowmatrix}) and the equation defines $F$.

Let $(f_0,\ldots,f_{2N-1})$ denote the components of $F$. We have
\begin{equation}\label{linearfk}
\begin{split}
  f_k & =l_k^\beta(y_{k+1}^r-y_k^r)+l_{k-1}^\beta(y_{k-1}^r-y_k^r), \\
  f_{N+k}  & =l_k^\beta(y_{k+1}^i-y_k^i)+l_{k-1}^\beta(y_{k-1}^i-y_k^i),
\end{split}
\end{equation} where $l_k$ is the distance between $(y_k^r,y_k^i)$ and $(y_{k+1}^r,y_{k+1}^i)$ for $k=0,\ldots,N-1$.
We can now describe the linearization.

\begin{theorem}\label{linearrealthm}
 The linearization of (\ref{realselfauto}) around the regular $N$-gon $P_1$ (\ref{eigvecofM}) is
\begin{equation}\label{linearreal}
  \frac{dY}{d\tau}=-\lambda_1Y+\begin{pmatrix}
                                    M & \mathbf{0} \\
                                    \mathbf{0} & M \\
                                  \end{pmatrix}Y+\beta\begin{pmatrix}
                                    A & C \\
                                    C & B \\
                                  \end{pmatrix} Y.
\end{equation} Here $M$ is defined in (\ref{linearM}), $\mathbf{0}$ is the $N\times N$ 0-matrix, and the nonzero entries of $A=(a_{ij})$, $B=(b_{ij})$, and $C=(c_{ij})$ are:
\begin{displaymath}
  \begin{split}
     & a_{kk\pm 1}=\sin^2(2k \pm 1)\theta,\quad  a_{kk}=-[\sin^2(2k-1)\theta+\sin^2(2k+1)\theta],
%     ,\quad a_{kk+1}=\sin^2(2k+1)\theta,
     \\
      & b_{kk \pm 1}=\cos^2(2k \pm 1)\theta,\quad  b_{kk}=-[\cos^2(2k-1)\theta+\cos^2(2k+1)\theta],
%      ,\quad b_{kk+1}=\cos^2(2k+1)\theta,
\\
      & c_{kk \pm 1}=-\cos(2k \pm 1)\theta\sin(2k \pm 1)\theta, \quad
%      c_{kk+1}=-\cos(2k+1)\theta\sin(2k+1)\theta,
	\\
      &
      c_{kk}=\cos(2k-1)\theta\sin(2k-1)\theta+\cos(2k+1)\theta\sin(2k+1)\theta,
  \end{split}
\end{displaymath}
%$$A=\left(
%                      \scalemath{0.8}{ \begin{array}{cccccc}
%                         -[\sin^2(-\theta)+\sin^2\theta] & \sin^2\theta & 0 & \cdots & 0 & \sin^2(-\theta) \\
%                         \sin^2\theta & -[\sin^2\theta+\sin^2(3\theta)] & \sin^2(3\theta) & 0 & \ddots & 0 \\
%                         0 & \ddots & \ddots & \ddots & 0 & \vdots \\
%                         \vdots & 0 & \sin^2\theta_{k-1} & -[\sin^2\theta_{k-1}+\sin^2\theta_k] & \sin^2\theta_k & 0 \\
%                         0 & \ddots & 0 & \ddots & \ddots & \ddots \\
%                         \sin^2\theta_{N-1} & 0 & \cdots & 0 & \sin^2\theta_{N-2} & -[\sin^2\theta_{N-2}+\sin^2\theta_{N-1}] \\
%                       \end{array}}
%                     \right),$$
for $k=0,\cdots,N-1$, where $\theta = \pi / N$. \end{theorem}
\begin{proof} Let $P_1=(x_0^r,\cdots,x_{N-1}^r,x_0^i,\cdots,x_{N-1}^i)$ be the regular $N$-gon defined in (\ref{eigvecofM}). We have $x_k^r  =\cos2k\theta$, $x_k^i   =\sin2k\theta,$ and $l_k=l=2\sin\theta$ for $k=0,\cdots,N-1.$

We can now differentiate (\ref{linearfk}) and evaluate at $P_1$. For instance, we have
\begin{displaymath}
\begin{split}
   \frac{\partial f_k}{\partial y_{k-1}^r}& =l^\beta[1+\beta \sin^2(2k-1)\theta], \\
%    \frac{\partial f_k}{\partial y_{k+1}^r} & =l^\beta[1+\beta\sin^2(2k+1)\theta],\\
    \frac{\partial f_k}{\partial y_{k}^r}&=-l^\beta\left[2+\beta\big(\sin^2(2k-1)\theta+\sin^2(2k+1)\theta\big)\right],\\
    \frac{\partial f_k}{\partial y_{k-1}^i}&=-\beta l^\beta\cos(2k-1)\theta\sin(2k-1)\theta,
%   \frac{\partial f_k}{\partial y_{k+1}^i}&=-\beta l^\beta\cos(2k+1)\theta\sin(2k+1)\theta,\\
%    \frac{\partial f_k}{\partial y_k^i}&=\beta l^\beta[\cos(2k-1)\theta\sin(2k-1)\theta+\cos(2k+1)\theta\sin(2k+1)\theta],
\end{split}
\end{displaymath}
%and
%\begin{displaymath}
%  \begin{split}
%    \frac{\partial f_{N+k}}{\partial y_{k-1}^r} & = -\beta l^\beta\cos(2k-1)\theta\sin(2k-1)\theta,\\
%     \frac{\partial f_{N+k}}{\partial y_{k+1}^r} & =-\beta l^\beta\cos(2k+1)\theta\sin(2k+1)\theta,\\
%     \frac{\partial f_{N+k}}{\partial y_{k}^r}&=\beta l^\beta[\cos(2k-1)\theta\sin(2k-1)\theta+\cos(2k+1)\theta\sin(2k+1)\theta],\\
%     \frac{\partial f_{N+k}}{\partial y_{k-1}^i}&=l^\beta[1+\beta\cos^2(2k-1)\theta],\\
%     \frac{\partial f_{N+k}}{\partial y_{k+1}^i}&=l^\beta[1+\beta\cos^2(2k+1)\theta],\\
%     \frac{\partial f_{N+k}}{\partial y_{k}^i}&=-l^\beta\left[2+\beta\big(\cos^2(2k-1)\theta+\cos^2(2k+1)\theta\big)\right],
%  \end{split}
%\end{displaymath}
for $k=0,\ldots,N-1.$
%Therefore, we have the following Jacobian matrix of $F$ at $P_1$:
%\begin{displaymath}
%  \left.\frac{\partial F}{\partial Y}\right\vert_{P_1}=l^\beta\begin{pmatrix}
%                                   M & 0 \\
%                                    0 & M \\
%                                  \end{pmatrix}+\beta l^\beta \begin{pmatrix}
%                                    A & C \\
%                                    C & B \\
%                                  \end{pmatrix},
%\end{displaymath} where $M,A,B,C$ are described as above, and this completes the proof.
\end{proof}

%???What happens if we rotate or translate??? remark??

\subsection{Stability}

To study the local stability of (\ref{linearreal}) at $P_1$, it is enough to classify the eigenvalues
and eigenspaces of the matrix
$$\left[-\lambda_1 I+\begin{pmatrix}
                                    M & 0 \\
                                    0 & M \\
                                  \end{pmatrix}\right]+\beta\begin{pmatrix}
                                    A & C \\
                                    C & B \\
                                  \end{pmatrix}=:D+\beta E,$$
where $I$ is the $2N\times 2N$ identity matrix.

First recall the analysis of the matrix $M$. From (\ref{eigvecofM}) and (\ref{eigofM}), the real and imaginary parts of $P_k$ form a set of real vectors $\{c_k,s_k\}$ that spans the eigenspace of $\lambda_k$, where
\begin{equation}\label{realeig}
  \begin{split}
    c_k & =(1,\cos 2k\theta,\cdots,\cos 2k(N-1)\theta)^T, \\
    s_k  & =(0,\sin 2 k\theta,\cdots,\sin 2k(N-1)\theta)^T,
  \end{split}
\end{equation} for $0\leq k\leq \lfloor N/2\rfloor$. Note that if $k=0$ or $k=N/2$, the eigenspace for $\lambda_k$ is spanned by $\{c_k\}$ and is one-dimensional. For all other $k$, $\{c_k,s_k\}$ is a basis and the eigenspace is two-dimensional.  Let $\mathbf{0}=(0,\cdots,0)^T$ and $\mathbf{1}=(1,\cdots,1)^T$ denote vectors in $\mathbb{R}^N$. A straightforward calculation gives the following.
\begin{lemma}\label{eigenspaceofD}
  The $2N\times2N$ matrix $D=-\lambda_1 I+\begin{pmatrix}
                                    M & \mathbf{0} \\
                                    \mathbf{0} & M \\
                                  \end{pmatrix}$
                                  has the following eigenvectors：
  \begin{equation}\label{eigofD}
    \begin{pmatrix}
      c_k \\
      \mathbf{0} \\
    \end{pmatrix},\quad \begin{pmatrix}
      s_k \\
      \mathbf{0} \\
    \end{pmatrix},\quad\begin{pmatrix}
      \mathbf{0} \\
       c_k\\
    \end{pmatrix},\quad \begin{pmatrix}
      \mathbf{0} \\
       s_k\\
    \end{pmatrix},
  \end{equation} which span the eigenspace of $\lambda_k-\lambda_1$ for $0\leq k\leq \lfloor N/2\rfloor$.
  In particular, we have:
  \begin{enumerate}
  \item There is only one positive eigenvalue $-\lambda_1$ of $D$, and the eigenspace of $D$
  corresponding to it is spanned by
\begin{equation}\label{euclidtrans}
  \begin{pmatrix}
      \mathbf{1} \\
      \mathbf{0} \\
    \end{pmatrix},\quad \begin{pmatrix}
      \mathbf{0} \\
       \mathbf{1}\\
    \end{pmatrix}.
\end{equation}
\item The eigenspace of $D$ corresponding to the eigenvalue 0 is spanned by
\begin{equation}\label{lineartransp1}
  \begin{pmatrix}
      c_1 \\
      \mathbf{0} \\
    \end{pmatrix},\quad \begin{pmatrix}
      s_1 \\
      \mathbf{0} \\
    \end{pmatrix},\quad\begin{pmatrix}
      \mathbf{0} \\
       c_1\\
    \end{pmatrix},\quad \begin{pmatrix}
      \mathbf{0} \\
       s_1\\
    \end{pmatrix}.
\end{equation}
\item The other eigenvalues of $D$ are negative.
\end{enumerate}
\end{lemma}
The eigenvectors corresponding to the positive eigenvalue correspond to Euclidean translations in $\mathbb{C}^N$, while the
eigenvectors corresponding to the eigenvalue 0 correspond to linear transformations of the regular $N$-gon $P_1$.

We will use the following lemma to analyze the definiteness of the matrix $E$.
\begin{lemma}\label{lapcycle}
  Let $A$ be a matrix of the form:
  $$A=\left(
                       \begin{array}{cccccc}
                         -(a_0+a_{n-1}) & a_{0} & 0 & \cdots & 0 & a_{n-1} \\
                         a_{0} & -(a_{0}+a_{1}) & a_{1} & 0 & \ddots & 0 \\
                         0 & \ddots & \ddots & \ddots & 0 & \vdots \\
                         \vdots & 0 & a_{k-1} & -(a_{k-1}+a_{k}) & a_{k}& 0 \\
                         0 & \ddots & 0 & \ddots & \ddots & \ddots \\
                         a_{n-1} & 0 & \cdots & 0 & a_{n-2} & -(a_{n-2}+a_{n-1}) \\
                       \end{array}
                     \right).$$
                     Then for any vectors $x=(x_0,\cdots,x_{n-1})$ and $y=(y_0,\cdots,y_{n-1})\in \mathbb{R}^n$, we have $$xAy^T=-\sum_{k=0}^{n-1}a_k(x_{k+1}-x_k)(y_{k+1}-y_k).$$ In particular, $A$ is negative semidefinite if $a_k\geq0$ for all $0\leq k\leq n-1$.
\end{lemma}
\begin{proof}
  This is a straightforward calculation:
  \begin{displaymath}
    \begin{split}
      xAy^T & =x\cdot\begin{pmatrix}
                       a_0(y_1-y_0)+a_{n-1}(y_{n-1}-y_0) \\
                       \vdots \\
                       a_{k-1}(y_{k-1}-y_k)+a_{k}(y_{k+1}-y_k) \\
                       \vdots \\
                       a_{n-2}(y_{n-2}-y_{n-1})+a_{n-1}(y_0-y_{n-1}) \\
                     \end{pmatrix}
       \\
%        & =\sum_{k=0}^{n-1}x_k\left[a_{k-1}(y_{k-1}-y_k)+a_{k}(y_{k+1}-y_k)\right]\\
%        &=\sum_{k=0}^{n-1}\left[a_{k}x_{k+1}(y_{k}-y_{k+1})+a_{k}x_k(y_{k+1}-y_k)\right]\\
        &=-\sum_{k=0}^{n-1}a_k(x_{k+1}-x_k)(y_{k+1}-y_k).
    \end{split}
  \end{displaymath}
\end{proof}

We now look at the matrix $E$.
\begin{lemma}\label{eigofE}
  The $2N\times 2N$ matrix $E=\begin{pmatrix}
                                A & C \\
                                C &B \\
                              \end{pmatrix}
  $ is negative semidefinite. Moreover, the 0-eigenspace consists of the vectors of the following form:
  \begin{displaymath}
    \{X\in\mathbb{R}^{2N}:\sin\left[(2k+1)\theta\right](x_{k+1}^r-x_k^r)=\cos\left[(2k+1)\theta\right](x_{k+1}^i-x_k^i)
    \text{ for all } k=0,\cdots,N-1\},
  \end{displaymath} where $X=(x_0^r,\cdots,x_{N-1}^r,x_0^i,\cdots,x_{N-1}^i)$ and $\theta=\pi/N.$
\end{lemma}
\begin{proof}
%  For any $X\in\mathbb{R}^{2N}$, we write $X=(X^r,X^i)=(x_0^r,\cdots,x_{N-1}^r,x_0^i,\cdots,x_{N-1}^i)$.
  We have
  \begin{displaymath}
    \begin{split}
      XEX^T & =(X^r,X^i)\begin{pmatrix}
                                A & C \\
                                C &B \\
                              \end{pmatrix}\begin{pmatrix}
                                             (X^r)^T \\
                                             (X^i)^T \\
                                           \end{pmatrix}
                               \\
        & =X^rA(X^r)^T+X^rC(X^i)^T+X^iC(X^r)^T+X^iB(X^i)^T\\
        & =X^rA(X^r)^T+2X^rC(X^i)^T+X^iB(X^i)^T,
    \end{split}
  \end{displaymath}where the last identity follows from the fact that $C$ is symmetric.

  Applying Lemma \ref{lapcycle} to compute  $X^rA(X^r)^T$, $X^iB(X^i)^T$, and $X^rC(X^i)^T$ and letting $a_k=\sin^2[(2k+1)\theta],\cos^2[(2k+1)\theta]$ and $-\cos[(2k+1)\theta]\sin[(2k+1)\theta]$, we have
  \begin{displaymath}
    \begin{split}
      XEX^T & =-\sum_{k=0}^{N-1}\left[\sin^2\big((2k+1)\theta\big)(x_{k+1}^r-x_k^r)^2+\cos^2\big((2k+1)\theta\big)(x_{k+1}^i-x_k^i)^2\right] \\
        & +\sum_{k=0}^{N-1}2\cos\big((2k+1)\theta\big)\sin\big((2k+1)\theta\big)(x_{k+1}^r-x_k^r)(x_{k+1}^i-x_k^i)\\
        &=-\sum_{k=0}^{n-1}\left[\sin\big((2k+1)\theta\big)(x_{k+1}^r-x_k^r)-\cos\big((2k+1)\theta\big)(x_{k+1}^i-x_k^i)\right]^2\leq0.
    \end{split}
  \end{displaymath}
\end{proof}

In order to understand $D+\beta E$, we need a better understanding of how the eigenspaces of $D$ and $E$ intersect.

\begin{lemma}\label{eigspaceofd+eta}
  Let $D_0$ and $E_0$ denote the 0-eigenspace of $D$ and $E$, respectively.
  Then we have if $N\geq 5$, \begin{displaymath}
    D_0\cap E_0=\{z=tiP_1\in\mathbb{C}^N:t\in\mathbb{R}\},
  \end{displaymath} and if $N=4$,  \begin{displaymath}
    D_0\cap E_0=\{z=tiP_1+s\overline{P_1}\in\mathbb{C}^N:t,s\in\mathbb{R}\}.
    \end{displaymath}
\end{lemma}
Note that $iP_1$ generates rotations of the polygon and $\overline{P_1}$ is the same polygon as $P_1$ but oriented in the
opposite direction. The span of $P_1$ and $\overline{P_1}$ generate all linear tranformations of $P_1$.
\begin{proof}
  Let $X^r=(x_0^r,\cdots,x_{N-1}^r),X^i=(x_0^i,\cdots,x_{N-1}^i)\in \mathbb{R}^N$, and suppose $X=(X^r,X^i)\in D_0\cap E_0$. Lemma \ref{eigenspaceofD} says there exist real numbers $a_{11},a_{12},a_{21}$ and $a_{22}$ such that
  \begin{displaymath}
  \left\{\begin{split}
    x_k^r & = a_{11}\cos(2k\theta)+a_{12}\sin(2k\theta), \\
    x_k^i  & =a_{21}\cos(2k\theta)+a_{22}\sin(2k\theta),
  \end{split}\right.
  \end{displaymath}for $k=0,\cdots,N-1$ and $\theta=\pi/N.$ Substituting this into the equation
  $$\sin\left[(2k+1)\theta\right](x_{k+1}^r-x_k^r)=\cos\left[(2k+1)\theta\right](x_{k+1}^i-x_k^i),$$
  we obtain the following $N$ linear equations for $a_{11},a_{12},a_{21},a_{22}$:
  \begin{equation}\label{inter0eig2}
    -\sin^2[(2k+1)\theta]a_{11}+\cos[(2k+1)\theta]\sin[(2k+1)\theta](a_{12}+a_{21})-\cos^2[(2k+1)\theta]a_{22}=0,
  \end{equation} for $k=0,\cdots,N-1.$
  By subtracting (\ref{inter0eig2}) with $k=0$ from the same equation with $k=N-1$, we find that
  \begin{equation}\label{inter0eigeqn}
    a_{12}+a_{21}=0.
  \end{equation}
  This reduces (\ref{inter0eig2}) to the following:
  \begin{equation}\label{inter0eig3}
       -a_{11}\sin^2[(2k+1)\theta]-a_{22}\cos^2[(2k+1)\theta]=0,
  \end{equation} for $k=0,\cdots,N-1.$
  Applying $k=0$ and $k=1$ to (\ref{inter0eig3}), we have
  \begin{equation}
    \left\{\begin{array}{c}
      -a_{11}\sin^2\theta-a_{22}\cos^2\theta=0, \\
      -a_{11}\sin^2[3\theta]-a_{22}\cos^2[3\theta]=0.
    \end{array}\right.
  \end{equation} The determinant of the corresponding matrix is $\cos^2(3\theta)\sin^2\theta-\cos^2\theta\sin^2(3\theta)=\cos^2(3\theta)\cos^2\theta[\tan^2\theta-\tan^2(3\theta)]\neq0$ for $N\geq 5$, and which gives
  \begin{equation}\label{inter0eigeqn=5}
    a_{11}=a_{22}=0,
  \end{equation}
  Combining (\ref{inter0eigeqn}) and (\ref{inter0eigeqn=5}), we have
   \begin{displaymath}
     \left\{\begin{split}
       x_k^r & =t\sin[(2k\theta)],\\
        x_k^i &=-t\cos[(2k\theta)],
     \end{split}\right.
   \end{displaymath} for some $t\in\mathbb{R}$. Viewed in $\mathbb{C}^N$, this gives
   \begin{displaymath}
    D_0\cap E_0=\{z=-tiP_1\in\mathbb{C}^N:t\in\mathbb{R}\}.
  \end{displaymath}
  %the space generated by the right angle rotation of $P_1$.

  For the case $N=4$, since $\sin^2[(2k+1)\pi/4]=\cos^2[(2k+1)\pi/4]$ for $k=0,\ldots,4,$ the system (\ref{inter0eig3}) reduces to
  \begin{equation}\label{inter0eigeqn=4}
    -a_{11}-a_{22}=0.
  \end{equation}
  Combining (\ref{inter0eigeqn}) and (\ref{inter0eigeqn=4}), we get
  \begin{displaymath}
    D_0\cap E_0=\{z=-tiP_1+s\overline{P_1}\in\mathbb{C}^N:t,s\in\mathbb{R}\}.
  \end{displaymath}
  %spanned by the right angle rotation and the backward of $P_1$.

\end{proof}
We now compute the stable, unstable, and center eigenspaces of the linearized system (\ref{linearreal}) at $P_1$. 

\begin{theorem}\label{eigspaceofd+e}
  Consider the $2N\times 2N$ matrix $D+\beta E=\left[-\lambda_1 I+\begin{pmatrix}
                                    M & 0 \\
                                    0 & M \\
                                  \end{pmatrix}\right]+\beta\begin{pmatrix}
                                    A & C \\
                                    C & B \\
                                  \end{pmatrix}$, where $A,B,C$ and $M$ are the matrices described in Theorem \ref{linearrealthm}. $D+\beta E$  has a positive eigenvalue $-\lambda_1$ with a two-dimensional eigenspace $E^u$ generated by the the vectors in (\ref{euclidtrans}). For $N\geq 5$, the 0-eigenspace $E^c$ is one dimensional, spanned by the vector $iP_1$; for $N=4$, the 0-eigenspace $E^c$ is two dimensional spanned by $\{iP_1, \overline{P_1}\}$. The remaining eigenvalues are negative. We call the span of these eigenspaces $E^s$.
\end{theorem}
\begin{proof}
  For any $a,b\in\mathbb{R}$, let $\mathbf{a}=(a,\cdots,a),\mathbf{b}=(b,\cdots,b)\in\mathbb{R}^N$. It is a straightforward calculation to check that $(D+\beta E)(\mathbf{a},\mathbf{b})^T=-\lambda_1(\mathbf{a},\mathbf{b})^T$. Let $V$ be the space generated by the vectors (\ref{euclidtrans}) and denote the orthogonal complement by $V^\perp$. Lemma \ref{eigenspaceofD} and Lemma \ref{eigofE} imply that $D+\beta E$ is negative semidefinite on $V^\perp$, and this implies that there cannot be any other positive eigenvalues. In particular, $-\lambda_1$ is the only positive eigenvalue and its eigenspace is $V$.

 Applying Lemma \ref{eigspaceofd+eta}, we obtain the 0-eigenspace $D_0\cap E_0$ for $D+\beta E$.

    Finally, on the orthogonal complement $W$ of $V\bigoplus(D_0\cap E_0)$, both $D$ and $E$ are negative semidefinite, but any nonzero vector in $W$ must miss either $D_0$ or $E_0$. Therefore, $D+\beta E$ is negative definite on $W$, which implies that all of the other eigenvalues of $D+\beta E$ are negative.
\end{proof}

%Theorem \ref{eigspaceofd+e} implies the following stability result.

The following result is an immediately consequence of Theorem \ref{eigspaceofd+e}.

\begin{theorem}\label{p1stable}
  Perturbations of $P_1$ that are orthogonal to $E^u$ and $E^c$ converge to $P_1$ when evolved by the linearized system (\ref{linearreal}).
\end{theorem}

We have the following stability result for the scaling system (\ref{turnselfauto}).

\begin{theorem}\label{constructsemistable}
  Assume $N\geq 5$. Around the regular polygon $P_1$ there exists a $2N-2$ dimensional semi-stable manifold $\mathcal{W}\subseteq E^c\bigoplus E^s$ such that for any $x_0\in\mathcal{W}$, the trajectory $x(t)$ of (\ref{turnselfauto}) with $x(0)=x_0$ converges to the regular polygon $e^{i\eta}P_1$ for some $\eta\in[0,2\pi).$ In particular, there exists an open neighborhood $U\subseteq \mathbb{R}^{2N}$ containing $P_1$ such that $U\cap(E^c\bigoplus E^s)\subseteq\mathcal{W}. $
\end{theorem}
\begin{proof}
  The center manifold theorem and Theorem \ref{p1stable} implies that there exists a $2N-3$ dimensional stable manifold $\mathcal{W}^s$ for (\ref{turnselfauto}) in some neighbourhood of $P_1$.

  First, we claim that $\mathcal{W}^s$ is orthogonal to $E^u$. This means that $\mathcal{W}^s$ is a hypersurface of $E^c\bigoplus E^s$. In fact, let $x_0\in\mathcal{W}^s$ and $x(\tau)$ be the trajectories of (\ref{turnselfauto}) starting from $x_0$. It is clear that the center of mass $q=\sum_{i=0}^{N-1}x_i/N$ satisfies the following evolution: $$\frac{dq}{d\tau}=-\lambda_1q.$$ So if $q(0)$ is not zero, eventually $q(\tau)$ must go to infinity and so it cannot converge to $P_1$. This gives $x_0\in (E^u)^\perp=E^c\bigoplus E^s.$

  The rotational invariance of (\ref{turnselfauto}) implies that the stable manifold associates to $e^{i\eta}P_1$ is $e^{i\eta}\mathcal{W}^s$ for any $\eta\in[0,2\pi)$ . Moreover, $e^{i\eta_1}\mathcal{W}^s\cap e^{i\eta_2}\mathcal{W}^s=\emptyset$ if $\eta_1\neq\eta_2$.

  Now, consider the disjoint union $$\mathcal{W}:=\coprod_{\eta\in[0,2\pi)} e^{i\eta}\mathcal{W}^s.$$ $\mathcal{W}$ is a $2N-2$ dimensional manifold since it homeomorphic to $\mathcal{W}^s\times S^1.$ Moreover, $\mathcal{W}\bot E^u$ since $e^{i\eta}\mathcal{W}^s\bot E^u$ for any $\eta$. For any $x_0\in\mathcal{W}$, there exists a $\eta\in[0,2\pi)$ such that $x_0\in e^{i\eta}\mathcal{W}^s$, and therefore, the corresponding trajectories must converge to $e^{i\eta}P_1.$
\end{proof}

%The invariant property Lemma \ref{leminvariant} allows us to classify the solutions according to its initial shape. In fact, for any regular polygon $c\cdot e^{i\eta}P_1$, follow the same process as we do for Theorem \ref{p1stable}, first, we obtain a scaling system (analogues to (\ref{turnselfauto}) ) which has $c\cdot e^{i\eta}P_1$ as its equilibrium point, next  we linearize the scaling system at $c\cdot e^{i\eta}P_1$ and obtain a linear system (analogous to (\ref{linearreal})), the corresponding eigenvalues and eigenspace obey the similar distribution as we get in Theorem \ref{eigspaceofd+e}, finally, similiar to Theorem \ref{constructsemistable}, we can construct a $2N-2$ dimensional semi-stable manifold around $c\cdot e^{i\eta}P_1$, therefore we obtain the similar local stability result at any regular polygon.
%\begin{theorem}\label{scalep1stable}
% Assume $N\geq 5$, around the regular polygon $c\cdot e^{i\eta}P_1$ there exists a $2N-2$ dimensional semi-stable manifold $\mathcal{W}\subseteq E^c\bigoplus E^s$ such that for any $x_0\in\mathcal{W}$, the trajectory $x(t)$ of the corresponding scaling system (analogous to system (\ref{turnselfauto})) with $x(0)=x_0$ converges to the regular polygon $c\cdot e^{i(\eta+\tilde{\eta})}P_1$ for some $\tilde{\eta}\in[0,2\pi).$
%\end{theorem}
%
%\begin{remark}
%For different $c$, similar to the derivation in (\ref{selfp1}), we are able to get a scaling function $a_c(t)$ with $a_c(0)=1$ and $\displaystyle\lim_{t\rightarrow\infty}a_c(t)=0.$
%\end{remark}

Collecting these results, we can prove Theorem \ref{stableshapethm5}.

\begin{proof}[Proof of Theorem \ref{stableshapethm5}]

   %Let  $X_0\in\mathbb{C}^N$ be some regular $N$-gon, i.e., there exists some $\alpha\in[0,2\pi)$, $T_1\in E^u$ and $c_1>0$ such that $X_0=e^{i\alpha}(c_1P_1)+T_1$ or $X_0=e^{i\alpha}(c_1\overline{P_1})+T_1$. Since (\ref{flowmatrix}) is invariant under conjugations, without loss of generous, we assume that $X_0=e^{i\alpha}(c_1P_1)+T_1$. Moreover, still by the rigid motion invariant property as mentioned in the beginning of section \ref{turnsec}, we assume that $\alpha=0, T_1=0$ and $c_1=1$, i.e., $X_0=P_1$.
   Let  $X_0\in\mathbb{C}^N$ be a regular $N$-gon, i.e., there exists a scaling $c$ and a Euclidean isometry $L$ such that $X_0=cLP_1$, where $P_1$ is the regular $N$-gon defined in (\ref{eigofM}).  Lemma \ref{leminvariant} and Theorem \ref{exist} show that $X_0$ and $P_1$ behave in a similar way under the evolution (\ref{flowmatrix}). Therefore, without loss of generality, we assume $X_0=P_1.$

  For any $\epsilon>0$ and perturbations $F\in\mathbb{C}^N$, we decompose $F$ into two parts and write $F=T+F^\perp$, for some $T\in E^u$ and $F^\perp\in E^c\bigoplus　E^s$. Let $X^\epsilon_0=X_0+\epsilon F=P_1+T_\epsilon+\epsilon F^\perp$, where $T_\epsilon=\epsilon T$. Let $X^\epsilon(t)$ and $\tilde{X}^\epsilon(t)$ be the solution of (\ref{flowmatrix}) with $X^\epsilon(0)=X^\epsilon_0$ and $\tilde{X}^\epsilon(0)=P_1+\epsilon F^\perp,$ respectively. Since the system (\ref{flowmatrix}) is invariant under translations, Theorem \ref{exist} implies that $X^\epsilon(t)=\tilde{X}^\epsilon(t)+T_\epsilon$ for all $t>0.$ Since $P_1+\epsilon F^\perp\in E^c\bigoplus　E^s$, Theorem \ref{constructsemistable} implies that $P_1+\epsilon F^\perp\in\mathcal{W}$ for sufficiently small $\epsilon$. Therefore, the solution $\tilde{Y}^\epsilon(t)$ of (\ref{turnselfauto}) with $\tilde{Y}^\epsilon(0)=\tilde{X}^\epsilon(0)$ converges to the regular $N$-gon $e^{i\eta}P_1$ for some $\eta\in[0,2\pi)$. Moreover, from the construction of (\ref{turnselfauto}), we know that $a(t)\tilde{Y}^\epsilon(t)$ is a solution of (\ref{flowmatrix}), where $a(t)$ is the scaling function we derived in Lemma \ref{selfp1lem}.

  Since $a(0)\tilde{Y}^\epsilon(0)=1\cdot\tilde{Y}^\epsilon(0)=\tilde{X}^\epsilon(0)$, by Theorem \ref{exist} we have $\tilde{X}^\epsilon(t)=a(t)\tilde{Y}^\epsilon(t)$ for all $t>0$. Since $\tilde{Y}^\epsilon(t)\rightarrow e^{i\eta} P_1$ as $t\rightarrow\infty$, $\displaystyle\lim_{t\rightarrow\infty}a(t)=0$, and $X^\epsilon(t)=\tilde{X}^\epsilon(t)+T_\epsilon$, it follows that $X^\epsilon(t)$ shrinks to a point as $t\rightarrow\infty$, and the limiting shape is a regular polygon.

  %For arbitrary $\alpha,T_1$ and $c_1$, follow the similar process and using theorem \ref{scalep1stable} instead of theorem \ref{constructsemistable}, we obtain the descried result.
\end{proof}

\subsection{Quadrilaterals}

%\begin{theorem}\label{stableshapethm4}
%  When $N=4,$ the shape of square is locally stable on a 7-dimensional hypersurface $\mathcal{W}'$ under the flow (\ref{flowmatrix}).
%\end{theorem}
In this section, we prove the weaker stability result for quadrilaterals.

\begin{proof}[Proof of Theorem \ref{stableshapethm4}]
  Let $X_0\in\mathbb{C}^4$ be a square. By the invariance property, Lemma \ref{leminvariant}, 
   we can assume $X_0=P_1$, where $P_1$ is the regular square defined in (\ref{eigofM}).
  Similar to Theorem \ref{constructsemistable}, we construct a 5-dimensional semi-stable manifold $\mathcal{W}$ orthogonal to $E^u$ such that for any $x_0\in\mathcal{W}$, the trajectory $x(t)$ of (\ref{turnselfauto}) with $x(0)=x_0$ converges to the square $e^{i\eta}P_1$ for some $\eta\in[0,2\pi).$
  
   Consider the disjoint union $$\mathcal{W}'=\coprod_{x\in E^u}(x+\mathcal{W}).$$ We see that $\mathcal{W}'$ is a 7-dimensional hypersurface of $\mathbb{R}^8$ since it homeomorphic to $\mathbb{R}^2\times \mathcal{W}.$
   Similar to the argument in Theorem \ref{stableshapethm5}, the results follows.
\end{proof}

We see that $N=4$ is exceptional, since around any regular square $P_1$, there exists a non-regular one-parameter family of self-similar rhombus solutions of the form $P_1+\epsilon \overline{P_1}.$ See Example \ref{example:rhombus}. Note that these are affinely-regular and equilateral, but not regular. In case of larger $N$, there do not exist 
affinely-regular but non-regular equilateral polygons. 

The main obstacle to a complete result for quadrilaterals is a clear description of the
linearization of the rescaled flow around a rhombus.

\end{document}